\documentclass[twocolumn]{autart}    

\usepackage{cite}
\usepackage{amsmath,amssymb,amsfonts}
\usepackage{graphicx}
\usepackage{algorithm,algorithmic}
\usepackage{hyperref}
\usepackage{textcomp}

\usepackage{xcolor}

    \usepackage{amsmath} 
    \usepackage{amssymb}  
    \usepackage{stmaryrd} 
    \usepackage{amssymb}  
    \usepackage{empheq} 

    \usepackage{graphicx} 
    \usepackage[hang,small,bf]{caption}
    \usepackage[subrefformat=parens]{subcaption}
    \usepackage{stfloats}
    \fnbelowfloat

    \newcommand{\vertiii}[1]{{\left\vert\kern-0.25ex\left\vert\kern-0.25ex\left\vert #1 
        \right\vert\kern-0.25ex\right\vert\kern-0.25ex\right\vert}}

    \newenvironment{proof}[1][Proof]{%
      \par\noindent\textbf{#1. }\ignorespaces
    }{%
      \hfill$\square$\par
    }
    
    \newcommand{\qedtheorem}{\hfill \ensuremath{\Diamond}}

\begin{document}

\begin{frontmatter}

\title{On Policy Stochasticity in Mutual Information Optimal Control of Linear Systems} 

\thanks[footnoteinfo]{This paper was not presented at any IFAC 
meeting. Corresponding author K. Kashima. 
}

\author[Kyoto]{Shoju Enami}\ead{enami.shoujyu.57r@st.kyoto-u.ac.jp},    
\author[Kyoto]{Kenji Kashima}\ead{kk@i.kyoto-u.ac.jp},               

\address[Kyoto]{Graduate School of Informatics, Kyoto University, Kyoto, Japan}  


\begin{abstract}                          
Mutual information regularization has recently been studied in reinforcement learning (RL) as an extension of entropy regularization, in which not only the policy but also the reference distribution (prior) is optimized, and has also been used for privacy preserving policy design.
To obtain theoretical insight into how this joint optimization affects policy stochasticity, we study a model-based linear-quadratic setting that preserves this key structure while permitting an explicit analysis.
Specifically, we consider a mutual information optimal control problem (MIOCP) for stochastic discrete-time linear systems with quadratic costs and Gaussian policies and priors.
We first review alternating optimization of the policy and the prior with a technical extension.
We then establish the existence of an optimal solution, characterize its covariance matrices, and derive sufficient conditions under which the optimal policy is either a stochastic feedback policy or a deterministic open-loop policy.
We further interpret this characterization in terms of the state information disclosed through the control input.
We also analyze the limiting behavior of alternating optimization under the sufficient conditions.
The validity of the theoretical results is demonstrated through numerical experiments.
\end{abstract}

\end{frontmatter}

\section{Introduction}\label{sec:Introduction}

{\color{blue}
Maximum entropy optimal control and reinforcement learning (RL) introduce stochasticity into the policy by adding an entropy regularization term to the objective function \cite{genewein2015bounded, haarnoja2018soft, haarnoja2017reinforcement, ito2023maximum, ito2024maximum}.
Entropy regularization has several advantages, such as promoting exploration in RL \cite{haarnoja2017reinforcement}. It encourages the policy to have higher entropy, which can also be viewed formally as reducing the Kullback-Leibler (KL) divergence from a uniform reference distribution.
Mutual information regularization has been studied as an extension of entropy regularization \cite{grau2018soft, leibfried2020mutual, malloy2020deep, enami2025mutual}.
In this approach, not only the feedback policy but also the state-independent (i.e., open-loop) reference policy is optimized simultaneously.
This reference policy is referred to as the \textit{prior} in previous studies \cite{grau2018soft, leibfried2020mutual, malloy2020deep, enami2025mutual}.
Because the prior is adapted jointly with the policy, the effect of the regularization on the resulting policy can differ qualitatively from that in maximum entropy RL.

In maximum entropy RL, increasing the regularization coefficient, often referred to as the \textit{temperature}, typically increases policy stochasticity \cite{haarnoja2017reinforcement}.
In contrast, under mutual information regularization, the prior itself changes with the temperature.
Consequently, structural properties of the optimal policy, such as whether it is stochastic or deterministic and whether it depends on the state, are not straightforward to characterize. Despite the growing use of mutual information regularization, these properties have received limited analytical study.

It is also important to understand how the temperature affects the policy obtained by numerical algorithms.
Existing algorithms in mutual information RL commonly optimize the policy and the prior alternatively.
In \cite{leibfried2020mutual}, convergence of such alternating optimization to an optimal solution is established under the restrictive assumption that the state distribution is independent of the policy.
However, this assumption does not generally hold in control problems, where the policy affects the state distribution through the system dynamics.
It therefore remains unclear how the temperature affects the policy obtained by alternating optimization when this dependence is taken into account.

Mutual information regularization also has a direct connection to privacy protection in policy design.
If the current state is regarded as sensitive information and an eavesdropper has access to the current control input but not to the state, the mutual information between the state and the input quantifies the information about the state disclosed through the control input.
Mutual information regularization has been used in policy optimization with privacy constraints \cite{cundy2024privacy}.
While that work develops an algorithm, it does not provide a detailed theoretical characterization of how mutual information regularization affects the structure of the resulting optimal policy.
This privacy interpretation provides another motivation for understanding how mutual information regularization affects the dependence of the optimal policy on the state, in addition to its stochasticity.

The central question of this paper is how the optimal policy and the policy obtained by alternating optimization change as the temperature parameter varies.
Model-based linear-quadratic control problems have been widely used as analytically tractable settings for obtaining theoretical insights into RL and policy optimization \cite{wang2020reinforcement, hambly2021policy, gravell2020learning, fazel2018global, guo2026fast}.
Following this line of research, we consider a model-based mutual information optimal control problem (MIOCP) in a tractable linear-quadratic setting in which the effect of jointly optimizing the policy and the prior can be characterized analytically.
As a preliminary step, we review the alternating optimization algorithm for the MIOCP while extending it to accommodate degenerate Gaussian priors, thereby allowing deterministic and stochastic policies to be treated within a unified formulation.
We then establish the existence of an optimal solution and provide a characterization of it.
Next, we derive sufficient conditions under which the optimal policy is a stochastic feedback policy and, in another regime, a deterministic open-loop policy.
Finally, we analyze the alternating optimization algorithm and characterize the properties of its limiting policies under the same conditions.}

\paragraph*{Organization}

This paper is organized as follows.
Section \ref{sec:Problem Formulation} formulates an MIOCP for stochastic discrete-time linear dynamics with quadratic costs, a Gaussian initial state distribution, and a Gaussian prior class.
Section \ref{sec:Alternating Optimization of the MIOCP} reviews the alternating optimization algorithm for the MIOCP with a technical extension.
Section \ref{sec:Properties of Optimal Solutions to the MIOCP} establishes the existence of an optimal solution, characterizes an optimal solution, and derives sufficient conditions on the temperature parameter under which the optimal policy becomes either a stochastic feedback policy or a deterministic open-loop policy.
Section \ref{sec:Properties of the Alternating Optimization Algorithm for the MIOCP} analyzes the limiting behavior of the alternating optimization algorithm and derives analogous conditions for the resulting policy.
Section \ref{sec:Numerical Examples} presents numerical examples to illustrate the theoretical results of alternating optimization.
Finally, Section \ref{sec:Conclusion} concludes the paper.

\paragraph*{Notation}
Define the imaginary unit as $\mathrm{i}:=\sqrt{-1}$.
The set of all integers that are larger than or equal to $a$ is denoted by $\mathbb{Z}_{\geq a}$.
The Borel $\sigma$-algebra on $\mathbb{R}^{n}$ is denoted by $\mathcal{B}_{n}$.
The set of integers $\{k,k+1,\ldots, l\}(k\leq l)$ is denoted by $\llbracket k, l\rrbracket$.
For two scalars $x,y \in \mathbb{R}$, denote the minimum and maximum functions by $\min(x,y)$ and $\max(x,y)$.
The set of all symmetric matrices of size $n$ is denoted by $\mathbb{S}^{n}$.
For $A, B \in \mathbb{S}^{n}$, we write $A \succ B$ (resp. $A \succeq B$) if $A-B$ is positive definite (resp. positive semidefinite).
The identity matrix is denoted by $I$, and its dimension depends on the context.
The Euclidean norm and the Frobenius norm are denoted by the same notation $\|\cdot \|$.
The determinant and the trace of $A \in \mathbb{R}^{n\times n}$ are denoted by $|A|$ and $\mathrm{Tr}(A)$, respectively.
For $A \in \mathbb{R}^{n\times m}$, denote the image of $A$ by $\mathrm{Im}(A)$.
For $x \in \mathbb{R}^{n}$ and $A \in \mathbb{S}^{n}$, denote $\|x\|_{A} := (x^{\top}Ax)^{\frac{1}{2}}$.
Note that $\|\cdot\|_{A}$ is not a norm unless $A \succ 0$.
For $A\in\mathbb{S}^{n}$, denote its smallest and largest eigenvalues by $\lambda_{\min}(A)$ and $\lambda_{\max}(A)$, respectively. 
For $A\in\mathbb{R}^{n\times m}$, denote its smallest and largest singular values by $\sigma_{\min}(A)$ and $\sigma_{\max}(A)$, respectively.
For $A \in \mathbb{R}^{n \times m}$, denote its Moore-Penrose inverse by $A^{\dagger}$.
For the eigenvalue decomposition of symmetric matrix $A=U\mathrm{diag}(\lambda_{1},\ldots,\lambda_{m})U^\top$, define its positive semidefinite part as $A_{+} := U\mathrm{diag}(\max(\lambda_{1},0),\ldots,\max(\lambda_{m},0))U^{\top}$.
The expected value of a random variable is denoted by $\mathbb{E}[\ \cdot\ ]$.
A multivariate Gaussian distribution on $\mathcal{B}_{n}$ with mean $\mu \in \mathbb{R}^{n}$ and covariance matrix $\Sigma \succeq 0$ is denoted by $\mathcal{N}(\mu,\Sigma)$.
Denote the probability density function (PDF) of $\mathcal{N}(\mu,\Sigma)$ by $\tilde{\mathcal{N}}(\mu,\Sigma)$ if it exists.
When we emphasize that a random variable $w\in \mathbb{R}^{n}$ follows $\mathcal{N}(\mu, \Sigma)$, $w$ is described explicitly as $\mathcal{N}(w | \mu, \Sigma)$.
For probability distributions $p$ and $q$, the Radon–Nikodym derivative is denoted by $\frac{dp}{dq}$ when it is defined.
The KL divergence between probability distributions $p$ and $q$ is denoted by $\mathcal{D}_{\text{KL}}[p \| q]$ when it is defined. The mutual information between two random variables $x,y$ is denoted by $\mathcal{I}(x;y)$.
We use the same symbol for a random variable and its realization.
We abuse the notation $p$ as the probability distribution of a random variable depending on the context.

\section{Problem Formulation} \label{sec:Problem Formulation}
In this paper, we investigate the following MIOCP.
\begin{prob}
    Find a pair of a policy $\pi = \{\pi_{k}\}_{k=0}^{T-1}$ and a prior $\rho = \{\rho_{k}\}_{k=0}^{T-1}$ that solves
    \begin{align}
        &\min_{\pi, \rho\in \mathcal{R}} J(\pi,\rho)\nonumber\\
        &:=\mathbb{E}\left[ \sum_{k=0}^{T-1} \left\{\frac{1}{2}\|u_{k}\|_{R_{k}}^{2} {\color{blue}+ \frac{1}{2}\|x_{k}\|_{Q_{k}}^{2}}\right.\right.\nonumber \\
        &\hspace{30pt}\left.\left.  + \varepsilon \mathcal{D}_{\text{KL}}[\pi_{k}(\cdot|x_{k}) \| \rho_{k}]  \right\}+ \frac{1}{2}\|x_{T}\|_{F}^{2} \right] \label{eq:objective function of MIOCP}\\
        &\mbox{s.t. }x_{k+1} = A_{k} x_{k} + B_{k} u_{k} + w_{k}, \label{eq:linear system of MIOCP}\\
        &\hspace{16pt}u_{k} \sim \pi_{k}(\cdot|x) \ \mbox{given }x=x_{k}, \label{eq:stochastic feedback input of MIOCP}\\
        &\hspace{16pt}w_{k} \sim \mathcal{N}(0,\Sigma_{w_{k}}), \label{eq:process noise}\\
        &\hspace{16pt}x_{0} \sim \mathcal{N}(0, \Sigma_{x_{\text{ini}}}), \label{eq:initial condition of MIOCP}
    \end{align}
    where $T \in \mathbb{Z}_{\geq 1}, x_{k} \in \mathbb{R}^{n}, u_{k} \in \mathbb{R}^{m}, A_{k} \in \mathbb{R}^{n \times n}, B_{k}\in \mathbb{R}^{n \times m}, Q_{k} \succeq 0, F\succeq 0, R_{k}\succ 0, \Sigma_{w_{k}}\succ 0, \Sigma_{x_{\text{ini}}} \succ 0$.
    The temperature parameter $\varepsilon >0$ determines the relative importance of the regularization term versus the cost.
    The prior class $\mathcal{R}$ is defined as
    \begin{align*}
        \mathcal{R} := \{ & \rho = \{ \rho_{k} \}_{k=0}^{T-1} \mid \rho_{k} = \mathcal{N}({\color{blue}0}, \Sigma_{\rho_{k}}), \Sigma_{\rho_{k}} \succeq 0 \}.
    \end{align*}
    A stochastic policy $\pi_{k}$ is a conditional probability measure on $\mathcal{B}_{m}$ given $x_{k} = x$ and a prior $\rho_{k}$ is a probability measure on $\mathcal{B}_{m}$. \qedtheorem
    \label{prob:MIOCP}
\end{prob}

{\color{blue}
We consider Problem \ref{prob:MIOCP} in a model-based setting to isolate the effect of mutual information regularization from learning and sampling effects in terms of RL.
We further focus on linear dynamics with standard quadratic state and input costs, which provide a tractable setting for analyzing the structure of the optimal policy.
To obtain an analytical characterization, we also restrict the prior to Gaussian distributions and consider the prior class $\mathcal{R}$.
}

We briefly clarify the relation of Problem \ref{prob:MIOCP} to maximum entropy optimal control.
By formally fixing $\rho_{k}$ to a uniform distribution $p^{\mathrm{uni}}(u)\propto 1$, Problem \ref{prob:MIOCP} reduces to the following maximum entropy optimal control problem (MEOCP) \cite{ito2023maximum}.
\begin{align*}
    &\min_{\pi} \mathbb{E}\left[ \sum_{k=0}^{T-1} \left\{\frac{1}{2}\|u_{k}\|_{R_{k}}^{2} + \frac{1}{2}\|x_{k}\|_{Q_{k}}^{2} - \varepsilon \mathcal{H}(\pi_{k}(\cdot|x_{k})) \right\} \right. \\
    &\left. \hspace{35pt}+ \frac{1}{2}\|x_{T}\|_{F}^{2} \right]\\
    &\mbox{s.t. }\eqref{eq:linear system of MIOCP}\text{--}\eqref{eq:initial condition of MIOCP},
\end{align*}
where $\mathcal{H}(p):=-\int_{\mathbb{R}^{m}}\pi_{k}(u|x_{k})\log \pi_{k}(u|x_{k}) du$ is the entropy of $\pi_{k}(\cdot|x_{k})$.
Thus, Problem \ref{prob:MIOCP} extends the MEOCP by optimizing the prior jointly with the policy.
Moreover,
\begin{align*}
    \min_{\rho_k} \mathbb{E}\!\left[ \mathcal{D}_{\mathrm{KL}} [\pi_k(\cdot|x_k)\|\rho_k] \right] = \mathcal{I}(x_k;u_k)
\end{align*}
\cite{grau2018soft, leibfried2020mutual, enami2025mutual}, which motivates the term \textit{mutual information optimal control} and provides the following privacy interpretation.
{\color{blue}Suppose that the state $x_{k}$ is regarded as sensitive information and that an eavesdropper knows the system model and the control policy and observes the current control input $u_{k}$, but not $x_{k}$.
Under this setting, $\mathcal{I}(x_{k};u_{k})$ quantifies the information about the current state disclosed through the control input.
Hence, Problem \ref{prob:MIOCP} can also be interpreted as balancing control performance against instantaneous information leakage, corresponding to the basic privacy setting considered in \cite{cundy2024privacy}.
}

\begin{rem}
    In this remark, we explain why $\rho$ is referred to as the ``prior''.
    As shown in \cite{grau2018soft}, Problem \ref{prob:MIOCP} with $\rho$ fixed is equivalent to an inference problem, known as ``control as inference'' \cite{levine2018reinforcement, rawlik2012stochastic}, which infers the probability of given state and input sequences being optimal, that is, minimizing the cost \eqref{eq:objective function of MIOCP} excluding the regularization term.
    The prior distribution in this inference problem is given by the distribution of the state and input sequences generated by the dynamics \eqref{eq:linear system of MIOCP}, the noise \eqref{eq:process noise}, the initial condition \eqref{eq:initial condition of MIOCP}, and the stochastic input given by $u_{k}\sim \rho_{k}$.
    For this reason, $\rho$ is literally called the prior.
    \qedtheorem
    \label{rem:origin of mutual information optimal control}
\end{rem}

\begin{rem}
    Problem \ref{prob:MIOCP} can be generalized as follows:
    \begin{align*}
        &\min_{\pi, \rho \in \mathcal{R}} \mathbb{E}\left[ \sum_{k=0}^{T-1} \left\{\frac{1}{2}\|u_{k}\|_{R_{k}}^{2} + \frac{1}{2}\|x_{k}\|_{Q_{k}}^{2} \right. \right. \\
        &\left.\left. \hspace{33pt} + \varepsilon \mathcal{D}_{\text{KL}}[\pi_{k}(\cdot|x_{k}) \| \rho_{k}] \right\}+ \frac{1}{2}\|x_{T}-\mu_{x_{\text{fin}}}\|_{F}^{2} \right] \\
        &\mbox{s.t. }\eqref{eq:linear system of MIOCP}\text{--}\eqref{eq:process noise}, x_{0} \sim \mathcal{N}(\mu_{x_{\text{ini}}}, \Sigma_{x_{\text{ini}}}),
    \end{align*}
    where $\mu_{x_{\text{ini}}}\in \mathbb{R}^{n}, \mu_{x_{\text{fin}}} \in \mathbb{R}^{n}$.
    By following the same argument as in \cite[Section IV]{ito2023maximum}, this generalized MIOCP can be actually decomposed into a linear-quadratic-regulator (LQR) problem and Problem \ref{prob:MIOCP}.
    The LQR problem can be solved by applying existing results such as \cite{lewis2012optimal}.
    We therefore focus on the MIOCP in the simple case given by Problem \ref{prob:MIOCP}.
    \label{rem:nonzero initial and target mean}
    \qedtheorem
\end{rem}

{\color{blue}
\begin{rem}
    One may consider the following prior class with nonzero mean.
    \begin{align*}
        \tilde{\mathcal{R}}:=\{ & \rho = \{ \rho_{k} \}_{k=0}^{T-1} \mid \rho_{k} = \mathcal{N}(\mu_{\rho_{k}}, \Sigma_{\rho_{k}}), \\
        &\mu_{\rho_{k}}\in \mathbb{R}^{m},\Sigma_{\rho_{k}} \succeq 0 \}.
    \end{align*}
    Actually, the optimal prior mean is given by $\mu_{\rho_{k}}=0$ and thus we consider the zero mean prior class $\mathcal{R}$.
    For the details, see Appendix \ref{app:Justification of Zero Mean Priors}.
    \label{rem:nonzero mean prior}
    \qedtheorem
\end{rem}
}

\section{Alternating Optimization} \label{sec:Alternating Optimization of the MIOCP}

This section reviews the alternating optimization algorithm for Problem \ref{prob:MIOCP} proposed in \cite{enami2025mutual}.
Although the flow in this section mirrors that in \cite{enami2025mutual}, there is a technical extension from the result in \cite{enami2025mutual}.
Specifically, the prior class $\mathcal{R}$ in this paper contains degenerate Gaussian distributions, whereas \cite{enami2025mutual} only considers nondegenerate Gaussian priors.
As a result, the results of \cite{enami2025mutual}, which assume the existence of PDFs of the policy and the prior, can not be directly used.
Although this extension is technical, it will play an important role in our main results presented in Sections \ref{sec:Properties of Optimal Solutions to the MIOCP} and \ref{sec:Properties of the Alternating Optimization Algorithm for the MIOCP}, as discussed in Remark \ref{rem:importance of the choice of prior class R}.

\subsection{Optimal Policy for Fixed Prior}\label{subsec:Optimal Policy for Fixed Priorr}

Let us introduce the following lemma.
\begin{lem} 
    For a given prior $\rho \in \mathcal{R},\rho_{k}=\mathcal{N}(0,\Sigma_{\rho_{k}})$, define $\Pi_{k}$ as the solution to the following Riccati equation:
    \begin{align}
        \Pi_{k} = & Q_{k}+ A_{k}^{\top} \Pi_{k+1} A_{k} -\frac{1}{\varepsilon}A_{k}^{\top} \Pi_{k+1} B_{k} \Sigma_{\rho_{k}}^{1/2} \nonumber \\
        &\times (I+\Sigma_{\rho_{k}}^{1/2}C_{k}\Sigma_{\rho_{k}}^{1/2} )^{-1} \nonumber\\
        &\times\Sigma_{\rho_{k}}^{1/2} B_{k}^{\top} \Pi_{k+1} A_{k},k \in \llbracket 0, T-1\rrbracket ,\label{eq:Riccati difference equation}\\
        \Pi_{T} = & F, \label{eq:terminal condition of Riccati difference equation}
    \end{align}
    where $C_{k}:=(R_{k}+B_{k}^{\top}\Pi_{k+1}B_{k})/\varepsilon, k \in \llbracket 0,T-1 \rrbracket$.
    Then $\Pi_{k}\succeq 0$ for any $k \in \llbracket 0,T \rrbracket$.
    \qedtheorem
    \label{lem:positive semidefiniteness of Pi}
\end{lem}

\begin{proof}
    From the Woodbury matrix identity \cite[Theorem 18.2.8.]{harville1997matrix}, \eqref{eq:Riccati difference equation} can be rewritten as
    \begin{align}
        \Pi_{k} = &A_{k}^{\top}\Pi_{k+1}^{1/2}\left\{ I+\Pi_{k+1}^{1/2}B_{k}\Sigma_{\rho_{k}}^{1/2}(\varepsilon I+\Sigma_{\rho_{k}}^{1/2}R_{k}\Sigma_{\rho_{k}}^{1/2})^{-1} \right.\nonumber\\
        &\left.\times \Sigma_{\rho_{k}}^{1/2}B_{k}^{\top}\Pi_{k+1}^{1/2}\right\}^{-1}\Pi_{k+1}^{1/2}A_{k} + Q_{k}.\label{eq:Woodbury matrix identity for Pi}
    \end{align}
    Because $\Pi_{T}=F\succeq 0, Q_{T-1}\succeq 0$ and the expression in the curly brackets in \eqref{eq:Woodbury matrix identity for Pi} is positive definite, $\Pi_{T-1}\succeq 0$.
    By applying this procedure recursively, we obtain the desired result.
\end{proof}

Note that $C_{k} \succ 0$ for any $k \in \llbracket 0,T-1 \rrbracket$ from Lemma \ref{lem:positive semidefiniteness of Pi}.
Now, the following proposition derives the optimal policy for a fixed prior.
See Appendix \ref{app:Proof of Proposition of optimal policy of MIOCP for fixed prior} for the proof.

\begin{prop}
    Consider a given prior $\rho \in \mathcal{R},\rho_{k}=\mathcal{N}(0,\Sigma_{\rho_{k}})$.
    Then, the unique optimal policy $\pi^{\rho}$ of Problem \ref{prob:MIOCP} with the prior fixed to the given $\rho$ is given by
    \begin{align}
        \pi_{k}^{\rho}(\cdot|x) = \mathcal{N}(\mu_{\pi_{k}^{\rho}}, \Sigma_{\pi_{k}^{\rho}}),  k \in \llbracket 0,T-1 \rrbracket,\label{eq:optimal policy of MIOCP for fixed prior}
    \end{align}
    where
    \begin{align}
        \Sigma_{\pi_{k}^{\rho}}:=&\Sigma_{\rho_{k}}^{1/2}(I+\Sigma_{\rho_{k}}^{1/2}C_{k}\Sigma_{\rho_{k}}^{1/2})^{-1}\Sigma_{\rho_{k}}^{1/2},\label{eq:covariance matrix of optimal policy of MIOCP for fixed prior}\\
        \mu_{\pi_{k}^{\rho}} := &-\frac{1}{\varepsilon}\Sigma_{\pi_{k}^{\rho}}B_{k}^{\top}\Pi_{k+1}A_{k}x.\label{eq:mean of optimal policy of MIOCP for fixed prior}
    \end{align}
    \qedtheorem
    \label{prop:optimal policy of MIOCP for fixed prior}
\end{prop}

{\color{blue}
Let us relate Proposition \ref{prop:optimal policy of MIOCP for fixed prior} to the optimal policies of a linear-quadratic-Gaussian (LQG) problem and an MEOCP.
For simplicity, suppose that $\Sigma_{\rho_k}\succ0$ for any $k \in \llbracket 0,T-1 \rrbracket$.
Then, \eqref{eq:covariance matrix of optimal policy of MIOCP for fixed prior} can be rewritten as
\begin{align*}
    \Sigma_{\pi_{k}^\rho} = \varepsilon\left(R_{k}+B_{k}^{\top}\Pi_{k+1}B_{k} +\varepsilon\Sigma_{\rho_{k}}^{-1} \right)^{-1}.
\end{align*}
Then, the mean of the optimal policy is
\begin{align*}
    \mu_{\pi_{k}^\rho} = -&\left( R_{k}+\varepsilon\Sigma_{\rho_{k}}^{-1} +B_{k}^{\top}\Pi_{k+1}B_{k} \right)^{-1} B_{k}^{\top}\Pi_{k+1}A_{k}x.
\end{align*}
Therefore, for a fixed nondegenerate prior, the mean of the optimal MIOCP policy coincides with the optimal LQG control law with the
modified input cost matrix $R_{k}+\varepsilon\Sigma_{\rho_{k}}^{-1}$.
Thus, the prior affects not only the covariance of the stochastic policy but also its feedback gain.
In the formal limit in which the prior approaches a uniform distribution, the additional term $\varepsilon\Sigma_{\rho_{k}}^{-1}$ vanishes, and the optimal policy reduces to that of the MEOCP \cite{ito2023maximum}.
Its mean then coincides with the standard LQG control law for the original quadratic cost.
Because the prior is jointly optimized and the feedback gain and the policy covariance depend on the prior in the MIOCP, it is nontrivial to characterize whether the optimal policy is stochastic or deterministic and state-dependent or state-independent.
}

\subsection{Optimal Prior for Fixed Policy}\label{subsec:Optimal Prior for Fixed Policy}

Introduce the following policy class.
\begin{align*}
    \mathcal{P}:=\{&\pi = \{\pi_{k}\}_{k=0}^{T-1}\mid \pi_{k}(\cdot|x)=\mathcal{N}(P_{k}x,\Sigma_{\pi_{k}}),\\
    &P_{k}\in \mathbb{R}^{m\times n},\Sigma_{\pi_{k}}\succeq 0,\mathrm{Im}(P_{k})\subset \mathrm{Im}(\Sigma_{\pi_{k}})\}.
\end{align*}
Note that $\pi^{\rho} \in \mathcal{P}$ holds for any $\rho \in \mathcal{R}$ from Proposition \ref{prop:optimal policy of MIOCP for fixed prior}.
In addition, let us denote the mean and covariance matrix of the state $x_{k}$ by $\mu_{x_{k}}$ and $\Sigma_{x_{k}}$, respectively.
From \eqref{eq:linear system of MIOCP}--\eqref{eq:initial condition of MIOCP}, $\mu_{x_{k}}$ and $\Sigma_{x_{k}}$ evolve as follows under $\pi\in \mathcal{P}, \pi_{k}(\cdot|x)=\mathcal{N}(P_{k}x ,\Sigma_{\pi_{k}})$.
\begin{align}
    \mu_{x_{k}} =& 0 , k \in \llbracket0,T\rrbracket \label{eq:mean of state} \\
    \Sigma_{x_{k+1}} = &(A_{k} + B_{k}P_{k})\Sigma_{x_{k}} (A_{k} + B_{k}P_{k})^{\top} + B_{k} \Sigma_{\pi_{k}} B_{k}^{\top}\nonumber\\
    &+\Sigma_{w_{k}}, k \in \llbracket 0, T-1 \rrbracket, \label{eq:evolution of covariance matrix of state}\\
    \Sigma_{x_{0}} =& \Sigma_{x_{\text{ini}}}.\label{eq:initial covariance matrix of state}
\end{align}
Then, the optimal prior for a fixed $\pi \in \mathcal{P}$ is given by the following proposition.
See Appendix \ref{app:Proof of Proposition of optimal prior for fixed policy} for the proof.

\begin{prop}
    Consider a given policy $\pi\in \mathcal{P},\pi_{k}(\cdot|x)=\mathcal{N}(P_{k}x,\Sigma_{\pi_{k}})$.
    Then, the unique optimal prior $\rho^{\pi}$ of Problem \ref{prob:MIOCP} with the policy fixed to the given $\pi$ is given by
    \begin{align}
        &\rho_{k}^{\pi} = \mathcal{N}(0, \Sigma_{\pi_{k}} + P_{k} \Sigma_{x_{k}} P_{k}^{\top}),k \in \llbracket 0, T-1 \rrbracket. \label{eq:optimal prior for fixed policy}
    \end{align}
    \qedtheorem
    \label{prop:optimal prior for fixed policy}
\end{prop}

\subsection{Alternating Optimization Algorithm} \label{subsec:Alternating Optimization of the MIOCP}
On the basis of Propositions \ref{prop:optimal policy of MIOCP for fixed prior} and \ref{prop:optimal prior for fixed policy}, the alternating optimization algorithm for Problem \ref{prob:MIOCP} is given as follows:

\begin{alg}
    \hspace{1pt}
    \begin{description}
        \item[Step 1] Initialize the prior $\rho^{(0)} \in \mathcal{R}_{\succ 0}$.
        \item[Step 2] Calculate the policy $\pi^{(i)} := \pi^{\rho^{(i)}}$. 
        \item[Step 3] Calculate the prior $\rho^{(i+1)} := \rho^{\pi^{(i)}}$ and go back to Step 2. \qedtheorem
    \end{description}
    \label{alg:alternating optimization algorithm for MIOCPs}
\end{alg}

Note that $\mathcal{R}_{\succ 0}\subset \mathcal{R}$ is defined as
\begin{align*}
    \mathcal{R}_{\succ 0}:=\{&\rho=\{\rho_{k}\}_{k=0}^{T-1}\mid \rho_{k}=\mathcal{N}(0,\Sigma_{\rho_{k}}),\Sigma_{\rho_{k}}\succ 0\}.
\end{align*}
From Propositions \ref{prop:optimal policy of MIOCP for fixed prior} and \ref{prop:optimal prior for fixed policy}, $\pi^{\rho} \in \mathcal{P}$ and $\rho^{\pi} \in \mathcal{R}$ holds for $\rho \in \mathcal{R}$ and $\pi \in \mathcal{P}$, respectively.
It hence follows that $\pi^{(i)} \in \mathcal{P}$ and $ \rho^{(i+1)} \in \mathcal{R}$ for any $i \in \mathbb{Z}_{\geq 0}$ due to $\rho^{(0)} \in \mathcal{R}$, and consequently $\pi^{(i)}$ and $\rho^{(i+1)}$ can be exactly computed in Steps 2 and 3 by Propositions \ref{prop:optimal policy of MIOCP for fixed prior} and \ref{prop:optimal prior for fixed policy}, respectively.

\begin{rem}
    In this remark, we discuss the choice of the initial prior \(\rho^{(0)}\).
    From Propositions \ref{prop:optimal policy of MIOCP for fixed prior} and \ref{prop:optimal prior for fixed policy}, it follows that
    \begin{align*}
        \mathrm{Im}\left(\Sigma_{\rho_{k}^{(0)}}\right)
        =
        \mathrm{Im}\left(\Sigma_{\pi_{k}^{(0)}}\right)
        =
        \mathrm{Im}\left(\Sigma_{\rho_{k}^{(1)}}\right)
        =
        \cdots
    \end{align*}
    for every $k\in\llbracket 0,T-1\rrbracket$, where $\Sigma_{\rho_{k}^{(i)}}$ and $\Sigma_{\pi_{k}^{(i)}}$ denote the covariance matrices of $\rho_{k}^{(i)}$ and $\pi_{k}^{(i)}$, respectively.
    Thus, the image of the covariance matrix is invariant throughout alternating optimization.
    In particular, if $\Sigma_{\rho_{k}^{(0)}}$ is singular, all subsequent covariance matrices remain singular with the same image and hence can not converge to a positive definite matrix.
    In other words, the iterates are confined to a proper face of the positive semidefinite cone determined by $\mathrm{Im}(\Sigma_{\rho_{k}^{(0)}})$.
    Therefore, to avoid such a restriction on the covariance matrices, we choose $\rho^{(0)} \in \mathcal{R}_{\succ 0}$.
    \qedtheorem
    \label{rem:initialize prior appropriately}
\end{rem}

\section{Properties of Optimal Solutions} \label{sec:Properties of Optimal Solutions to the MIOCP}

In this section, we provide properties of the optimal solution to Problem \ref{prob:MIOCP}.
To facilitate the analysis, we eliminate the decision variable $\pi$ by optimizing only $\pi$ for a fixed $\rho\in \mathcal{R}$.
From the proof of Proposition \ref{prop:optimal policy of MIOCP for fixed prior}, we can derive the value function $V(0,x)$, which is defined as \eqref{eq:definition of value function for k < T} and \eqref{eq:definition of value function for k = T}, by following the procedure to calculate \eqref{eq:arranged value function for k = T-1} recursively, and consequently we have
\begin{align}
    &\hspace{-10pt}J(\rho):=J(\pi^{\rho},\rho)\nonumber\\
    =& \mathbb{E}[V(0,x_{0})]\nonumber\\
    =&\frac{1}{2}\mathbb{E}\left[\|x_{0}\|_{\Pi_{0}}^{2} \right.\nonumber\\
    & \left.  + \sum_{k=0}^{T-1}\left\{ \varepsilon \log |I+\bar{\Sigma}_{\rho_{k}}^{\top} C_{k} \bar{\Sigma}_{\rho_{k}}| + \mathrm{Tr}[\Pi_{k+1}\Sigma_{w_{k}}]\right\}\right]\nonumber\\
    =&\frac{1}{2}\left[\mathrm{Tr}[\Pi_{0}\Sigma_{x_{\text{ini}}}] \right.\nonumber\\
        &\left.\hspace{15pt}+ \sum_{k=0}^{T-1} \varepsilon \log \frac{|\Sigma_{\rho_{k}}+\Gamma_{k}|}{|\Gamma_{k}|} + \mathrm{Tr}[\Pi_{k+1}\Sigma_{w_{k}}] \right]
\end{align}
where
\begin{align}
    \Gamma_{k}:=C_{k}^{-1} = \varepsilon (R_{k}+B_{k}^{\top}\Pi_{k+1}B_{k})^{-1} \label{eq:Gamma}
\end{align}
and $\bar{\Sigma}_{\rho_{k}}$ is given by the same way as \eqref{eq:decomposition of positive semidefinite matrix} and \eqref{eq:decomposition of positive semidefinite matrix whtn it is zero matrix}.
Noting that $\Gamma_{k}\succ 0$ due to $C_{k}\succ 0$, we have 
\begin{align}
    \left|I+\bar{\Sigma}_{\rho_{k}}^{\top}C_{k}\bar{\Sigma}_{\rho_{k}}\right|=\frac{|\Sigma_{\rho_{k}}+\Gamma_{k}|}{|\Gamma_{k}|} \label{eq:change of term of logarithm of fraction}
\end{align}
from the matrix determinant lemma \cite[Theorem 18.1.1]{harville1997matrix}.
Therefore, Problem \ref{prob:MIOCP} can be rewritten as follows.

\begin{prob}
    \begin{align}
        &\min_{(\Sigma_{\rho_{0}},\ldots,\Sigma_{\rho_{T-1}})\in \mathcal{M}_{T}} \check{J}(\Sigma_{\rho_{0}},\ldots,\Sigma_{\rho_{T-1}})\nonumber\\
        &:=\frac{1}{2}\left[\sum_{k=0}^{T}\mathrm{Tr}[\Pi_{k}\Sigma_{w_{k-1}}]+ \sum_{k=0}^{T-1} \varepsilon \log \frac{|\Sigma_{\rho_{k}}+\Gamma_{k}|}{|\Gamma_{k}|} \right]\label{eq:objective function of simplified MIOCP}\\
        &\mbox{s.t. }\eqref{eq:Riccati difference equation},\eqref{eq:terminal condition of Riccati difference equation}, \eqref{eq:Gamma},\nonumber
    \end{align}
    where $\Sigma_{w_{-1}}:=\Sigma_{x_{\mathrm{ini}}}$, $\mathbb{S}_{\succeq 0}^{m}:=\{\Sigma \in \mathbb{S}^{m}|\Sigma\succeq 0\}$ and
    \begin{align*}
        \mathcal{M}_{T}:=\mathbb{S}^{m}_{\succeq 0}\times \cdots \times \mathbb{S}^{m}_{\succeq 0}
    \end{align*}
     is the $T$-fold Cartesian product of $\mathbb{S}_{\succeq 0}^{m}$.
    \qedtheorem
    \label{prob:simplified MIOCP}
\end{prob}

In the rest of Section \ref{sec:Properties of Optimal Solutions to the MIOCP}, $\pi$ is implicitly given by $\pi=\pi^{\rho}$ henceforth, and consequently $\Sigma_{x_{k}}$ is evaluated under $\pi^{\rho}$.

\begin{rem}
    As noted at the beginning of Section \ref{sec:Alternating Optimization of the MIOCP}, in contrast to \cite{enami2025mutual}, this paper considers priors of degenerate Gaussian distributions.
    By this extension, the feasible region $\mathcal{M}_{T}$ of Problem \ref{prob:simplified MIOCP} is a closed set, which is the key to proving the existence of an optimal solution in Section \ref{subsec:Existence}.
    Furthermore, in the rest of this paper, it enables us to analyze whether the policy is stochastic or deterministic because we can consider a Dirac delta distribution as a degenerate Gaussian distribution with a zero covariance matrix.
    \qedtheorem
    \label{rem:importance of the choice of prior class R}
\end{rem}

\subsection{Existence} \label{subsec:Existence}

This subsection establishes the existence of an optimal solution to Problem \ref{prob:simplified MIOCP}.
As preparation, we introduce some lemmas.
See Appendices \ref{app:Proof of Lemma of inequalites}--\ref{app:Proof of Lemma that J hat is coercive} for the proofs of Lemmas \ref{lem:inequalities for condition where optimal covariance matrices are full-rank}--\ref{lem:J hat is coercive}, respectively.

\begin{lem}
    Define the solution $\check{\Pi}_{k}$ to the following Riccati equation.
    \begin{align}
        \check{\Pi}_{k} = & Q_{k}+ A_{k}^{\top} \check{\Pi}_{k+1} A_{k} -A_{k}^{\top} \check{\Pi}_{k+1} B_{k} \nonumber \\
        &\times (R_{k} + B_{k}^{\top} \check{\Pi}_{k+1} B_{k})^{-1} B_{k}^{\top} \check{\Pi}_{k+1} A_{k},\nonumber \\
        &k \in \llbracket 0, T-1\rrbracket ,\label{eq:Riccati difference equation for uniform distribution}\\
        \check{\Pi}_{T} = & F. \label{eq:terminal condition of Riccati difference equation for uniform distribution}
    \end{align}
    In addition, define $\hat{\Pi}_{k}$ recursively by
    \begin{align*}
        \hat{\Pi}_{k} = & Q_{k}+ A_{k}^{\top} \hat{\Pi}_{k+1} A_{k}, k \in \llbracket 0, T-1\rrbracket\\
        \hat{\Pi}_{T} = & F. 
    \end{align*}
    Then, the solution $\Pi_{k}$ to the Riccati equation \eqref{eq:Riccati difference equation} and \eqref{eq:terminal condition of Riccati difference equation} satisfies
    \begin{align}
        \hat{\Pi}_{k} \succeq \Pi_{k} \succeq \check{\Pi}_{k} \succeq 0\label{eq:inequality of Pi}
    \end{align}
    for any $k \in \llbracket 0,T \rrbracket$.
    In addition, $\Gamma_{k}$ satisfies
    \begin{align}
         \hat{\Gamma}_{k} \succeq \Gamma_{k} \succeq \check{\Gamma}_{k} \succ 0 \label{eq:inequality of Gamma}
    \end{align}
    for any $k \in \llbracket 0,T-1 \rrbracket$, where
    \begin{align*}
        \hat{\Gamma}_{k}:=& \varepsilon(R_{k} + B_{k}^{\top}\check{\Pi}_{k+1}B_{k})^{-1},\\
        \check{\Gamma}_{k} :=& \varepsilon (R_{k} + B_{k}^{\top}\hat{\Pi}_{k+1}B_{k})^{-1}.
    \end{align*}
    \qedtheorem
    \label{lem:inequalities for condition where optimal covariance matrices are full-rank}
\end{lem}

\begin{lem}
    The function $\check{J}$ is continuous on $\mathcal{M}_{T}$.
    \qedtheorem
    \label{lem:J hat is continuous}
\end{lem}

\begin{lem}
    The function $\check{J}$ is coercive on $\mathcal{M}_{T}$, that is, $\check{J}\rightarrow \infty$ if there exists at least one $k \in \llbracket 0,T-1 \rrbracket$ such that $\|\Sigma_{\rho_{k}}\| \rightarrow\infty$.
    \qedtheorem
    \label{lem:J hat is coercive}
\end{lem}

Now, combining Lemmas \ref{lem:J hat is continuous} and \ref{lem:J hat is coercive} with \cite[Theorem 4.7]{andreasson2020introduction}, we obtain the following proposition.

\begin{prop}
    Problem \ref{prob:simplified MIOCP} has at least one optimal solution.
    \label{prop:existence of optimal solution}
    \qedtheorem
\end{prop}

{\color{blue}
\subsection{Characterization} \label{subsec:Characterization}

In Section \ref{subsec:Characterization}, we derive a first-order characterization of the covariance matrices of an optimal prior.
To derive a first-order optimality condition, we show the directional derivative of $\check{J}$ in the following lemma.
For the proof, see Appendix \ref{app:Proof of Lemma of derivative of J_check}.

\begin{lem}
    The directional derivative of $\check{J}$ at $(\bar{\Sigma}_{\rho_{0}},$ $\ldots,\bar{\Sigma}_{\rho_{T-1}}) \in \mathcal{M}_{T}$ in a direction $(S_{0}-\bar{\Sigma}_{\rho_{0}},\ldots,S_{T-1}-\bar{\Sigma}_{\rho_{T-1}})$ is given by
    \begin{align}
        &\lim_{t\rightarrow +0} \left\{\check{J}(\bar{\Sigma}_{\rho_{0}}+t(S_{0}-\bar{\Sigma}_{\rho_{0}}),\ldots,\right. \nonumber\\
        &\left. \bar{\Sigma}_{\rho_{T-1}}+t(S_{T-1}-\bar{\Sigma}_{\rho_{T-1}}))-\check{J}(\bar{\Sigma}_{\rho_{0}},\ldots,\bar{\Sigma}_{\rho_{T-1}})\right\}/t\nonumber\\
        &=\sum_{k=0}^{T-1}\mathrm{Tr}\left[\check{J}_{k}^{\prime}(\bar{\Sigma}_{\rho_{0}},\ldots,\bar{\Sigma}_{\rho_{T-1}})(S_{k}-\bar{\Sigma}_{\rho_{k}})\right],  \label{eq:directional derivative of J_check}
    \end{align}
    where $(S_{0},\ldots,S_{T-1}) \in \mathcal{M}_{T}$ and $\check{J}_{k}^{\prime}:\mathcal{M}_{T}\rightarrow \mathbb{S}^{m}$,
    \begin{align}
        &\check{J}_{k}^{\prime}(\Sigma_{\rho_{0}},\ldots,\Sigma_{\rho_{T-1}})\nonumber\\
        &:= \frac{\varepsilon}{2}L_{k}(\Sigma_{\rho_{k}}+\Gamma_{k}-E_{k}\Sigma_{x_{k}}E_{k}^{\top})L_{k}, \label{eq:derivative of J_check}
    \end{align}
    with
    \begin{align}
        E_{k}:=&\Gamma_{k} B_{k}^{\top} \Pi_{k+1} A_{k}/ \varepsilon, \label{eq:E} \\
        L_{k}:=&(\Gamma_{k} + \Sigma_{\rho_{k}})^{-1}\succ 0.\label{eq:L}
    \end{align}
    \label{lem:derivative of J_check}
    \qedtheorem
\end{lem}

On the basis of Lemma \ref{lem:derivative of J_check}, the following theorem provides a first-order optimality condition of Problem \ref{prob:simplified MIOCP}.
For the proof, see Appendix \ref{app:Proof of Theorem of characterization of optimal solution}.

\begin{thm}
    Define
    \begin{align*}
        \Xi_{k}^{*} := \Gamma_{k}^{-1/2}E_{k}\Sigma_{x_{k}}E_{k}^{\top}\Gamma_{k}^{-1/2},
    \end{align*}
    where $\Gamma_{k}$, $E_{k}$, and $\Sigma_{x_{k}}$ are evaluated at an optimal solution $\{\Sigma_{\rho_{k}^{*}}\}_{k=0}^{T-1}$.
    Then, for any $k\in\llbracket0,T-1\rrbracket$,
    \begin{align}
        \Sigma_{\rho_{k}}^{*} = \Gamma_{k}^{1/2}(\Xi_{k}^{*}-I)_{+} \Gamma_{k}^{1/2}. \label{eq:characterization of optimal solution}
    \end{align}
    \qedtheorem
    \label{thm:characterization of optimal solution}
\end{thm}

Theorem \ref{thm:characterization of optimal solution} shows a characterization of the optimal prior covariance.
However, the characterization is implicit because the right-hand side of \eqref{eq:characterization of optimal solution} depends on the optimal solution itself.
Nevertheless, the theorem provides information on the stochasticity of the optimal policy.
Indeed, by Propositions \ref{prop:optimal policy of MIOCP for fixed prior} and \ref{prop:optimal prior for fixed policy}, any optimal solution $(\pi^{*},\rho^{*})$ satisfies $ \mathrm{Im}(\Sigma_{\pi_{k}^{*}}) = \mathrm{Im}(\Sigma_{\rho_{k}^{*}}), k\in\llbracket0,T-1\rrbracket$.
Hence, Theorem \ref{thm:characterization of optimal solution} implies
\begin{align}
    \mathrm{rank}(\Sigma_{\pi_{k}^{*}}) = \mathrm{rank}(\Sigma_{\rho_{k}^{*}}) = \#\{j:\lambda_{j} (\Xi_{k}^{*})>1\}, \label{eq:optimal policy rank}
\end{align}
where $\lambda_{j}(\Xi_{k}^{*})$ denotes the eigenvalue of $\Xi_{k}^{*}$ and $\#\{j:\lambda_{j}(\Xi_{k}^{*})>1\}$ denotes the cardinality.
In particular, $\Xi_{k}^{*}\succ I$ corresponds to a nondegenerate stochastic policy, whereas $\Xi_{k}^{*}\preceq I$ corresponds to a deterministic policy.
If $\Xi_{k}^{*}$ has both eigenvalues larger than one and eigenvalues not larger than one, the optimal covariance is singular and nonzero.
Thus, in the intermediate regime, stochasticity can be retained only in a subspace of the input space.

Since $\Xi_{k}^{*}$ can not generally be evaluated without knowing the optimal solution, the characterization above does not provide an a priori condition in terms of the parameters of Problem \ref{prob:MIOCP}.
In the next subsection, we therefore derive directly verifiable sufficient conditions for the two extreme cases of a positive definite and a zero optimal covariance.
}

\subsection{Relation with the Temperature Parameter}\label{subsec:Relation with the Temperature Parameter}

Section \ref{subsec:Relation with the Temperature Parameter} derives sufficient conditions on $\varepsilon$ for the two extreme cases: $\Sigma_{\rho_{k}^{*}}\succ 0$ and $\Sigma_{\rho_{k}^{*}}= 0$.

\subsubsection{Sufficient Condition for Stochastic Optimal Policies}

We derive a sufficient condition where $\Sigma_{\rho_{k}^{*}}\succ 0$.
On the basis of Lemmas \ref{lem:inequalities for condition where optimal covariance matrices are full-rank} and \ref{lem:derivative of J_check}, we obtain the following theorem.

{\color{blue}
\begin{thm}
    Assume that $B_{k}^{\top}\check{\Pi}_{k+1}B_{k}\succ0$ and $A_{k}$ is invertible for any $k\in\llbracket0,T-1\rrbracket$.
    Define
    \begin{align}
        \underline{e}_{k} &:= \frac{\sigma_{\min}(A_{k})\lambda_{\min}(B_{k}^{\top}\check{\Pi}_{k+1}B_{k})}{\sigma_{\max}(B_{k})\lambda_{\max}(R_{k}+B_{k}^{\top}\hat{\Pi}_{k+1}B_{k})},\nonumber\\
        \alpha_{k} &:= \lambda_{\min}(\Sigma_{w_{k-1}}) \underline{e}^{2},\nonumber
    \end{align}
    If we choose $\varepsilon$ such that
    \begin{align}
        \check{M}_{k} &:= \alpha_{k} I - \varepsilon (R_{k}+B_{k}^{\top}\check{\Pi}_{k+1}B_{k})^{-1} \succ 0 \label{eq:condition where optimal covariance matrices are full-rank}
    \end{align}
    for any $k  \in \llbracket 0,T-1 \rrbracket$,
    then any optimal solution $\{\Sigma_{\rho_{k}^{*}}\}_{k=0}^{T-1}$ to Problem \ref{prob:simplified MIOCP} satisfies that $\Sigma_{\rho_{k}^{*}}\succ 0$ for any $k \in \llbracket 0,T-1 \rrbracket$.
    \qedtheorem
    \label{thm:condition where optimal covariance matrices are full-rank}
\end{thm}
}

See Appendix \ref{app:Proof of Theorem of condition where optimal covariance matrices are full-rank} for the proof.
Note that the definition of $\Xi_{k}^{*}$ gives
\begin{align*}
    \Xi_{k}^{*}-I = \Gamma_{k}^{-1/2}\left( E_{k}\Sigma_{x_{k}}E_{k}^{\top}-\Gamma_{k}\right)\Gamma_{k}^{-1/2}.
\end{align*}
In addition, as shown in the proof of Theorem \ref{thm:condition where optimal covariance matrices are full-rank}, the assumption $\check{M}_{k}\succ 0$ guarantees
\begin{align*}
    E_{k}\Sigma_{x_{k}}E_{k}^{\top}-\Gamma_{k} \succeq\check{M}_{k}\succ0.
\end{align*}
Hence, the assumption $\check{M}_{k}\succ0$ implies $\Xi_{k}^{*}\succ I$.
Theorem \ref{thm:condition where optimal covariance matrices are full-rank} says that $\pi^{*}$ is stochastic and state-dependent if
$\varepsilon$ is sufficiently small so that \eqref{eq:condition where optimal covariance matrices are full-rank} is satisfied.
When $\varepsilon$ is sufficiently small, reducing the quadratic cost becomes more important relative to reducing the mutual information cost.
A singular prior covariance restricts the input to a proper subspace of $\mathbb{R}^{m}$, which may prevent the policy from decreasing the state quadratic cost.
Therefore, small $\varepsilon$ yields a positive definite prior covariance and thus an optimal policy also has a positive definite covariance matrix.

\subsubsection{Sufficient Condition for Deterministic Optimal Policies} 
Contrary to Theorem \ref{thm:condition where optimal covariance matrices are full-rank}, we will show that $\Sigma_{\rho_{k}^{*}}=0$ when $\varepsilon$ is sufficiently large.

{\color{blue}
\begin{thm}
    Define the covariance matrix of the state with a zero control input $u_{k}=0,k\in \llbracket0,T-1 \rrbracket$ as
    \begin{align*}
        \Sigma_{x_{k+1}}^{\mathrm{zero}} =& A_{k}\Sigma_{x_{k}}^{\mathrm{zero}}A_{k}^{\top} + \Sigma_{w_{k}}, k \in \llbracket0,T-1 \rrbracket,\\
        \Sigma_{x_{0}}^{\mathrm{zero}}=& \Sigma_{x_{\text{ini}}}.
    \end{align*}
    In addition, define
    \begin{align*}
        \bar{e}_{k} &:= \frac{\sigma_{\max}(A_{k})\sigma_{\max}(B_{k})\lambda_{\max}(\hat{\Pi}_{k+1})}{\lambda_{\min}(R_{k}+B_{k}^{\top}\check{\Pi}_{k+1}B_{k})},\\
        \beta_{k} &:= \lambda_{\max}(\Sigma_{x_{k}}^{\mathrm{zero}})\bar{e}_{k}^{2}.
    \end{align*}
    If we choose $\varepsilon$ such that
    \begin{align}
        \hat{M}_{k} := \beta_{k}I - \varepsilon (R_{k}+B_{k}^{\top}\hat{\Pi}_{k+1}B_{k})^{-1}\prec 0 \label{eq:condition where optimal covariance matrices are 0}
    \end{align}
    for any $k \in \llbracket 0,T-1 \rrbracket$, then the optimal solution $\{\Sigma_{\rho_{k}^{*}}\}_{k=0}^{T-1}$ to Problem \ref{prob:simplified MIOCP} is unique and given by $\Sigma_{\rho_{0}^{*}}= \cdots = \Sigma_{\rho_{T-1}^{*}} =0$.
    \qedtheorem
    \label{thm:condition where optimal covariance matrices are 0}
\end{thm}
}

For the proof, see Appendix \ref{app:Proof of Theorem of condition where optimal covariance matrices are 0}.
The condition in Theorem \ref{thm:condition where optimal covariance matrices are 0} has the opposite interpretation to that in Theorem \ref{thm:condition where optimal covariance matrices are full-rank}.
As $\varepsilon$ becomes large, the KL cost dominates the objective, causing the policy $\pi$ to approach the open-loop prior $\rho$.
Consequently, the optimal policy begins to behave like an open-loop policy.
Since the terms other than the KL cost in \eqref{eq:objective function of MIOCP} are quadratic and the system is linear, the open-loop policy that minimizes the quadratic terms is trivially deterministic.
Therefore, when $\varepsilon$ is large, the optimal policy is expected to be a deterministic open-loop policy.
If the system \eqref{eq:linear system of MIOCP} is unstable, an open-loop policy can not regulate the state covariance $\Sigma_{x_{k}}$, resulting in a large terminal cost $\mathbb{E}[\frac{1}{2}\|x_{T}\|_{F}^{2}]$.
However, when $\varepsilon$ is sufficiently large such that minimizing the KL cost takes priority over reducing the terminal cost, the optimal policy becomes a deterministic open-loop policy.

{\color{blue}
\subsubsection{Effect of Joint Prior Optimization on the Temperature Dependence}
\label{subsubsec:Effect of Joint Prior Optimization on the Temperature Dependenc}

We now summarize how joint optimization of the prior changes the dependence of the optimal policy on the temperature parameter.
The optimal policy approaches the deterministic LQG feedback controller as $\varepsilon\rightarrow 0$, while Theorem \ref{thm:condition where optimal covariance matrices are full-rank} guarantees a stochastic feedback policy for sufficiently small \(\varepsilon>0\).
In contrast, Theorem \ref{thm:condition where optimal covariance matrices are 0} shows that sufficiently large $\varepsilon$ yields a deterministic open-loop policy.
This behavior differs from the MEOCP, where increasing $\varepsilon$ promotes policy stochasticity.
In the intermediate regime, Theorem \ref{thm:characterization of optimal solution} shows that the optimal covariance can be singular and nonzero, with its rank determined by the eigenvalues of $\Xi_{k}^{*}$ relative to one.
As illustrated in Fig. \ref{fig:rough_sketch}, unlike standard entropy regularization, policy stochasticity in the MIOCP need not vary monotonically with the temperature parameter.
}

\begin{figure}[htbp]
    \begin{center}
    \centerline{\includegraphics[width=85mm]{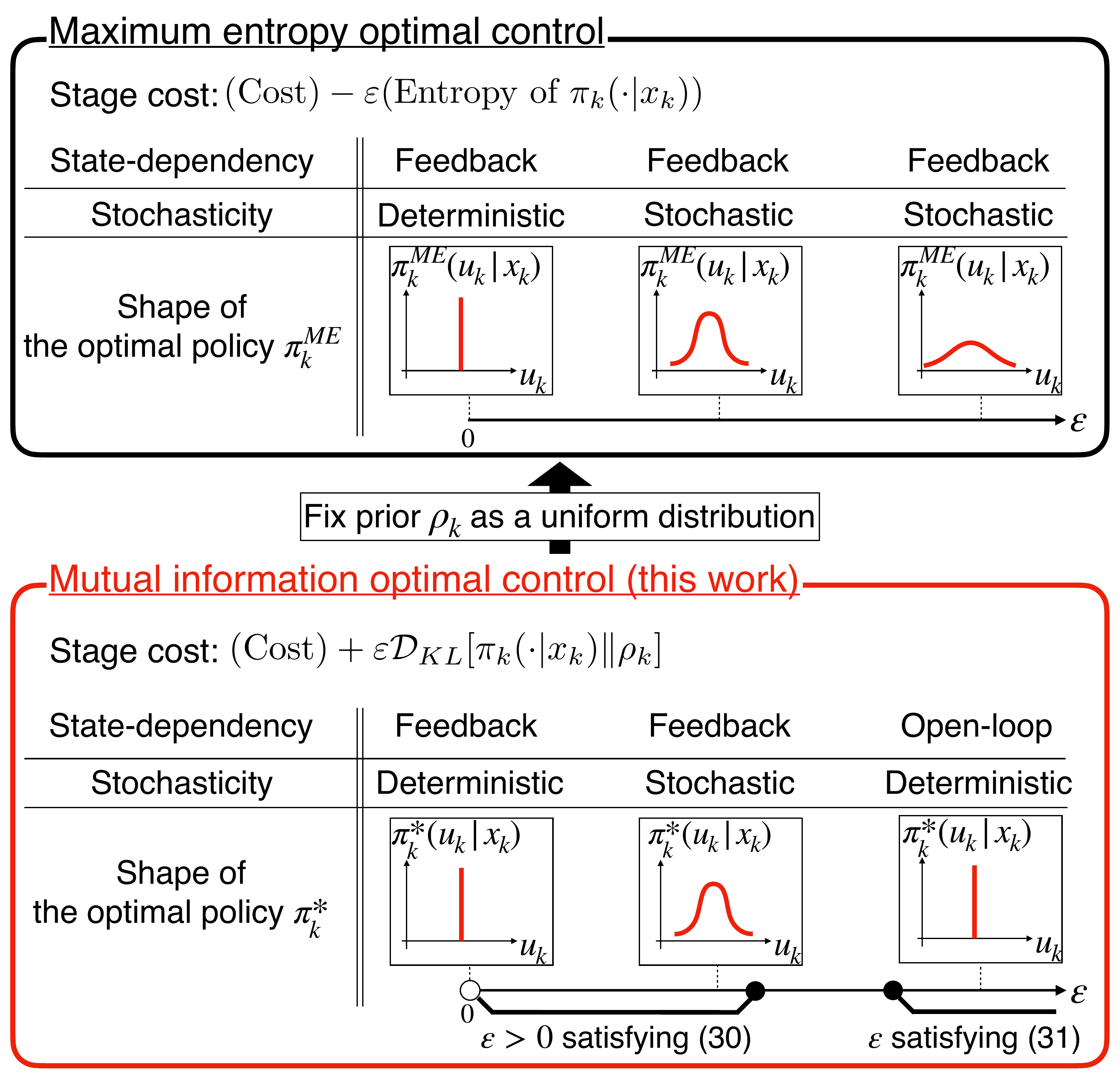}}
    \caption{Rough sketch of how the optimal policy $\pi_{k}^{ME}$ (in maximum entropy optimal control) and the optimal policy $\pi_{k}^{*}$ (in mutual information optimal control) relate to the temperature parameter $\varepsilon$.}
    \label{fig:rough_sketch}
    \end{center}
\end{figure}

{\color{blue}
\subsubsection{Implications for Mutual Information Regularization}

The preceding results can also be interpreted directly in terms of the mutual information between the state and the control input.
The following corollary follows from Theorem \ref{thm:characterization of optimal solution}.
See Appendix \ref{app:Proof of Corollary of mutual information at optimal solution} for the proof.

\begin{cor}
    Consider an optimal solution $(\pi^{*},\rho^{*})$ of Problem \ref{prob:MIOCP}.
    Under the optimal solution,
    \begin{align*}
        \mathcal{I}(x_{k};u_{k}) = \frac{1}{2} \sum_{j:\lambda_j(\Xi_{k}^{*})>1} \log \lambda_{j}(\Xi_{k}^{*}), k \in \llbracket 0,T-1 \rrbracket.
    \end{align*}
    \qedtheorem
    \label{cor:mutual information at optimal solution}
\end{cor}

Corollary \ref{cor:mutual information at optimal solution} shows that the eigenvalues in Theorem \ref{thm:characterization of optimal solution} characterize not only the stochasticity of the optimal policy but also the amount of state information contained in the control input.
In addition, note that the total mutual information $\sum_{k=0}^{T-1}\mathcal{I}(x_{k};u_{k})$ under an optimal policy is trivially nonincreasing with $\varepsilon$, although the policy stochasticity need not be monotone.

The fact that mutual information is nonincreasing with $\varepsilon$, whereas policy stochasticity need not, has two implications.
From the perspective of mutual information RL, the temperature should not be regarded simply as a parameter for increasing policy stochasticity, unlike in maximum entropy RL.
Increasing $\varepsilon$ in the same manner as in entropy regularization can eventually eliminate, rather than increase, policy stochasticity.
From the viewpoint of privacy protection, increasing $\varepsilon$ does not increase the information conveyed from the state to the control input.
Corollary \ref{cor:mutual information at optimal solution} further shows that only the directions associated with eigenvalues of $\Xi_{k}^{*}$ larger than one contribute to the mutual information.
Thus, in the intermediate regime, the reduction of information leakage can be accompanied by a restriction of the input subspace through which state information is conveyed.
Theorem \ref{thm:condition where optimal covariance matrices are 0} characterizes the extreme case, where sufficiently strong regularization completely removes the dependence of the control input on the state and achieves zero information leakage.
}

\section{Properties of the Alternating Optimization Algorithm} \label{sec:Properties of the Alternating Optimization Algorithm for the MIOCP}

In this section, we show that the limiting policy computed by Algorithm \ref{alg:alternating optimization algorithm for MIOCPs} is stochastic under the assumptions of Theorem \ref{thm:condition where optimal covariance matrices are full-rank} and deterministic under those of Theorem \ref{thm:condition where optimal covariance matrices are 0}.
Consequently, this policy is also state-dependent and state-independent under these conditions.

We provide a general property of Algorithm \ref{alg:alternating optimization algorithm for MIOCPs}.
Let us define a map $\mathcal{A}:\mathcal{R} \rightarrow \mathcal{R}, \rho \mapsto \rho^{+}=\arg\min_{\rho^{+}\in \mathcal{R}}J(\pi^{\rho},\rho^{+})$.
Note that $\mathcal{A}$ satisfies $\rho^{(i+1)}=\mathcal{A}(\rho^{(i)})$ for the sequence $\{\rho^{(i)}\}_{i\in \mathbb{Z}_{\geq 0}}$ generated by Algorithm \ref{alg:alternating optimization algorithm for MIOCPs}.
Using this notation, we have the following proposition.
For the proof, see Appendix \ref{app:Proof of Proposition of algorithm converges to equilibrium point}.

\begin{prop}
    The set $\mathcal{E}$ of all cluster points of the sequence $\{\rho^{(i)}\}_{i \in \mathbb{Z}_{\geq 0}}$ generated by Algorithm \ref{alg:alternating optimization algorithm for MIOCPs} satisfies $\mathcal{E}\subset\{ \rho \in \mathcal{R} | \rho = \mathcal{A}(\rho)\}$.
    \label{prop:algorithm converges to equilibrium point}
    \qedtheorem
\end{prop}

Proposition \ref{prop:algorithm converges to equilibrium point} ensures that Algorithm \ref{alg:alternating optimization algorithm for MIOCPs} converges to the set of fixed points of Algorithm \ref{alg:alternating optimization algorithm for MIOCPs}.
By \eqref{eq:covariance matrix of optimal policy of MIOCP for fixed prior} and \eqref{eq:optimal prior for fixed policy}, a fixed point $\rho \in \mathcal{R},\rho_{k}=\mathcal{N}(0,\Sigma_{\rho_{k}})$ satisfies
\begin{align}
    &\mathcal{A}(\rho)=\rho \nonumber\\
    &\Leftrightarrow\Sigma_{\rho_{k}}L_{k} \left(E_{k}\Sigma_{x_{k}}E_{k}^{\top} -  \Sigma_{\rho_{k}} -\Gamma_{k}\right)L_{k}\Sigma_{\rho_{k}}=0,\label{eq:equation that equilibrium points satisfy}\\
    &k \in \llbracket0,T-1 \rrbracket.\nonumber
\end{align}
{\color{blue}The above equation also provides a partial characterization of the covariance matrix of a fixed point.
Define $Y_{k}:=\Gamma_{k}^{-1/2}\Sigma_{\rho_{k}}\Gamma_{k}^{-1/2}, \Xi_{k}:=\Gamma_{k}^{-1/2}E_{k}\Sigma_{x_{k}}E_{k}^{\top}\Gamma_{k}^{-1/2}$ for a fixed point $\{\Sigma_{\rho_{k}}\}_{k=0}^{T-1}$.
Then, \eqref{eq:equation that equilibrium points satisfy} is equivalent to
\begin{align}
    Y_{k}(I+Y_{k})^{-1}(\Xi_{k}-I-Y_{k})(I+Y_{k})^{-1}Y_{k}=0. \label{eq:normalized fixed point condition}
\end{align}
Let $\mathsf{P}_{k}$ denote the orthogonal projection onto $\mathrm{Im}(Y_{k})$.
Since $Y_{k}(I+Y_{k})^{-1}$ has the same image as $Y_{k}$ and is positive definite on $\mathrm{Im}(Y_{k})$, \eqref{eq:normalized fixed point condition} implies
\begin{align*}
    \mathsf{P}_{k}(\Xi_{k}-I-Y_{k})\mathsf{P}_{k}=0. 
\end{align*}
It hence follows that
\begin{align}
    \mathsf{P}_{k}(\Xi_{k}-I)\mathsf{P}_{k} = Y_{k}. \nonumber
\end{align}
Therefore, a fixed point satisfies
\begin{align}
    \mathrm{rank}(\Sigma_{\rho_{k}}) = \mathrm{rank}(Y_{k}) \leq \#\{j:\lambda_{j}(\Xi_{k})>1\}. \label{eq:fixed point rank bound}
\end{align}
This characterization is weaker than that for an optimal solution in \eqref{eq:optimal policy rank}.
As in the case of an optimal solution, an a priori characterization of the covariance in the intermediate regime is difficult because the relevant quantities depend on the unknown fixed point.
Accordingly, we restrict our analysis below to directly verifiable sufficient conditions for the two extreme cases: a nonzero limiting covariance and a zero limiting covariance.
}

We first show that the policy computed by Algorithm \ref{alg:alternating optimization algorithm for MIOCPs} converges to a stochastic feedback one under the same assumptions as Theorem \ref{thm:condition where optimal covariance matrices are full-rank}.

\begin{thm}
    Suppose the same assumptions as Theorem \ref{thm:condition where optimal covariance matrices are full-rank}.
    If we choose $\varepsilon$ such that $\check{M}_{k}\succ 0$ for any $k \in \llbracket0,T-1 \rrbracket$, then the sequence $\{\rho^{(i)}\}_{i\in \mathbb{Z}_{\geq 0}}$ generated by Algorithm \ref{alg:alternating optimization algorithm for MIOCPs} converges to $\{\rho \in \mathcal{E}\mid \rho_{k}=\mathcal{N}(0,\Sigma_{\rho_{k}}), \Sigma_{\rho_{k}}\neq 0, k \in \llbracket0,T-1 \rrbracket\}$ in the sense of the distance to the set. 
    \label{thm:condition where algorithm converges to nonzero covariance matrix}
    \qedtheorem
\end{thm}

For the proof, see Appendix \ref{app:Proof of Theorem of condition where algorithm converges to nonzero covariance matrix}.
Note that Theorems \ref{thm:condition where optimal covariance matrices are full-rank} and \ref{thm:condition where algorithm converges to nonzero covariance matrix} are slightly different: Theorem \ref{thm:condition where optimal covariance matrices are full-rank} shows that an optimal covariance matrix is positive definite, whereas Theorem \ref{thm:condition where algorithm converges to nonzero covariance matrix} shows that the limiting covariance matrix computed by Algorithm \ref{alg:alternating optimization algorithm for MIOCPs} is not a zero matrix.

Next, we show that the policy computed by Algorithm \ref{alg:alternating optimization algorithm for MIOCPs} converges to a deterministic open-loop one when $\varepsilon$ is chosen as specified in Theorem \ref{thm:condition where optimal covariance matrices are 0}.
For the proof, see Appendix \ref{app:Proof of Theorem of condition where algorithm converges to zero covariance matrix} .
\begin{thm}
    If we choose $\varepsilon$ such that $\hat{M}_{k}\prec 0$ for any $k \in \llbracket0,T-1 \rrbracket$, then the sequence $\{\rho^{(i)}\}_{i\in \mathbb{Z}_{\geq 0}}$ generated by Algorithm \ref{alg:alternating optimization algorithm for MIOCPs} converges to $\{\mathcal{N}(0,0)\}_{k=0}^{T-1}$. 
    \qedtheorem
    \label{thm:condition where algorithm converges to zero covariance matrix}
\end{thm}

\section{Numerical Examples}\label{sec:Numerical Examples}

In this section, we demonstrate the validity of Theorems \ref{thm:condition where algorithm converges to nonzero covariance matrix} and \ref{thm:condition where algorithm converges to zero covariance matrix} through some numerical examples of Algorithm \ref{alg:alternating optimization algorithm for MIOCPs}.
The terminal time is given by $T=5$.
The system is given by
\begin{align}
    A_{k}=\begin{bmatrix}
        0.9 & 0.2\\ 0.1& 1.1
    \end{bmatrix},
    B_{k}=\begin{bmatrix}
        0.2 & 0 \\ 0.03 & 0.08
    \end{bmatrix}
    , \Sigma_{w_{k}} = 10^{-3}I\ \forall k.\nonumber
\end{align}
The covariance matrix of the initial state distribution and the coefficient matrices in \eqref{eq:objective function of MIOCP} are given by
\begin{align}
    \Sigma_{x_{\text{ini}}}=\begin{bmatrix}
        7 & 3\\ 3 & 5
    \end{bmatrix},F=10I, Q_{k}=0.5I, R_{k}=\begin{bmatrix}
        1 &0\\ 0&2
    \end{bmatrix} \ \forall k. \nonumber
\end{align}
Note that the above setting satisfies the assumptions of Theorem \ref{thm:condition where optimal covariance matrices are full-rank}, that is, $B_{k}^{\top}\check{\Pi}_{k+1}B_{k}\succ 0$ and $A_{k}$ is invertible for any $k \in \llbracket0,T-1\rrbracket$.
The initialized prior $\rho^{(0)},\rho_{k}^{(0)}(\cdot)=\mathcal{N}(0, \Sigma_{\rho_{k}^{(0)}})$ in Algorithm \ref{alg:alternating optimization algorithm for MIOCPs} is given by $ \Sigma_{\rho_{k}^{(0)}}=I \ \forall k$.

\begin{figure*}[htbp]
    \centering
    \begin{tabular}{@{}cc@{}}

      \begin{minipage}[t]{0.49\textwidth}
        \centering
        \includegraphics[width=0.95\linewidth]{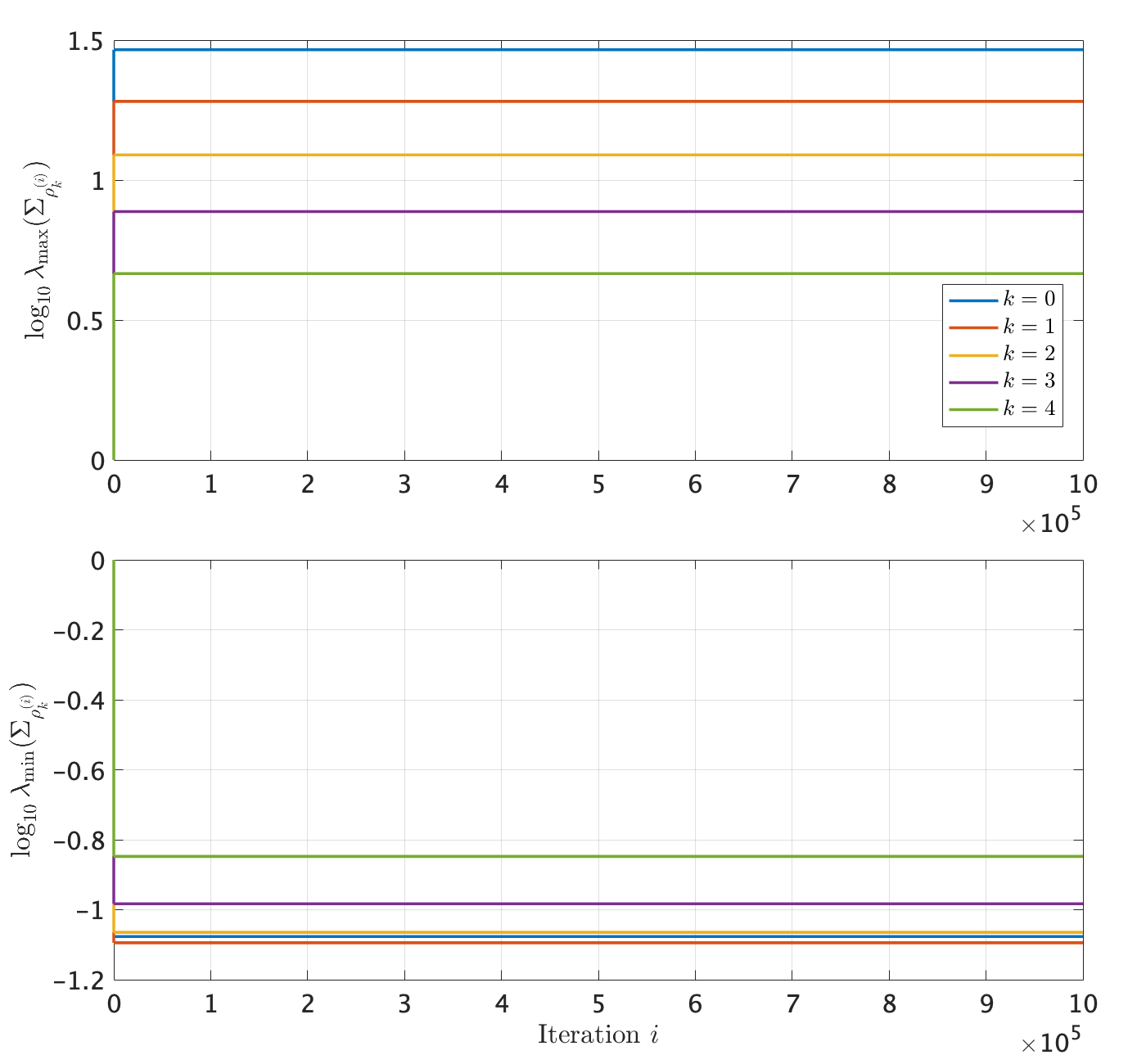}
        \subcaption{$\varepsilon = 10^{-6}$}
        \label{fig:eps10m6}
      \end{minipage}
      &
      \begin{minipage}[t]{0.49\textwidth}
        \centering
        \includegraphics[width=0.95\linewidth]{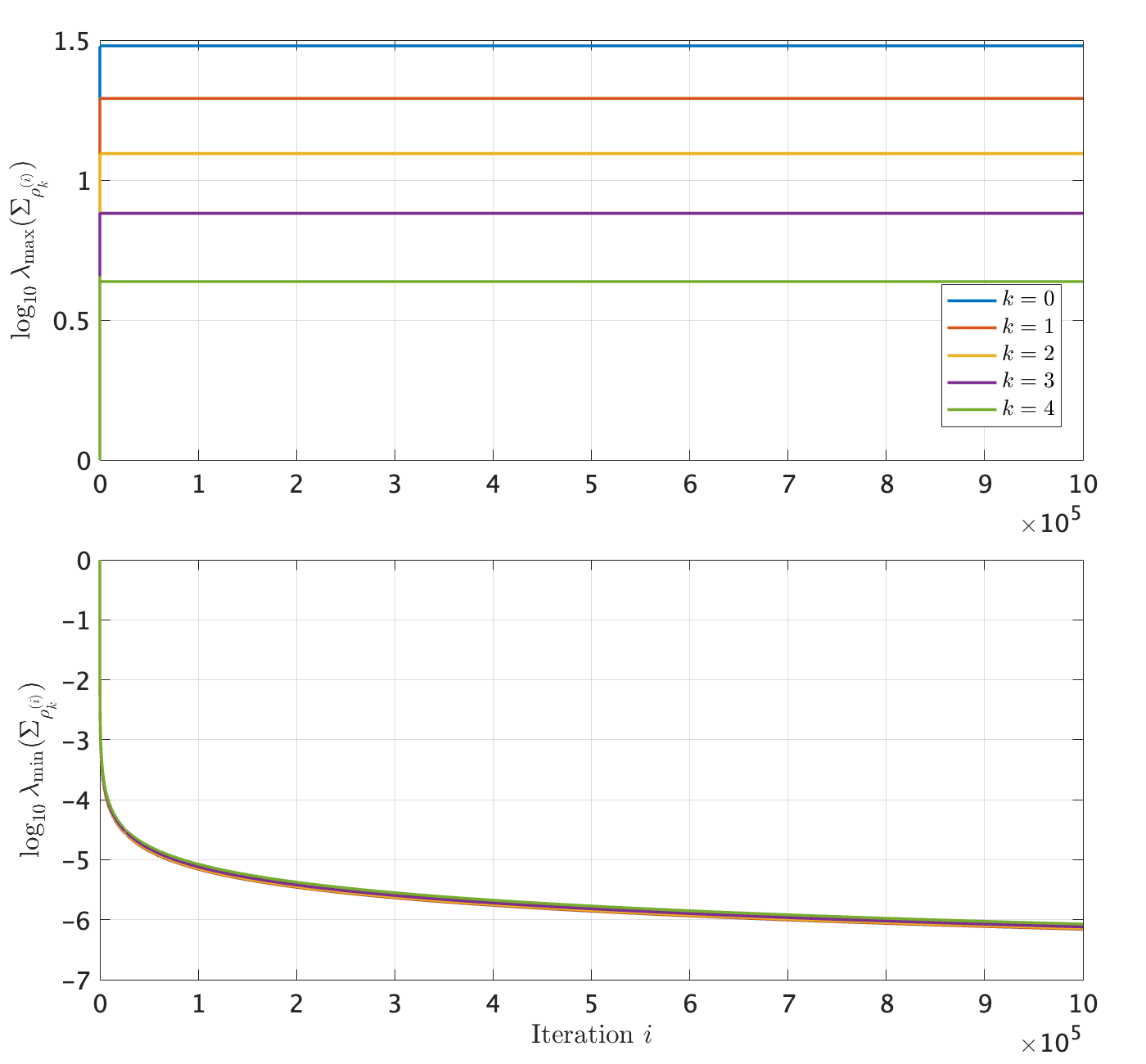}
        \subcaption{$\varepsilon = 1$}
        \label{fig:eps1}
      \end{minipage}
      \\

      \begin{minipage}[t]{0.49\textwidth}
        \centering
        \includegraphics[width=0.95\linewidth]{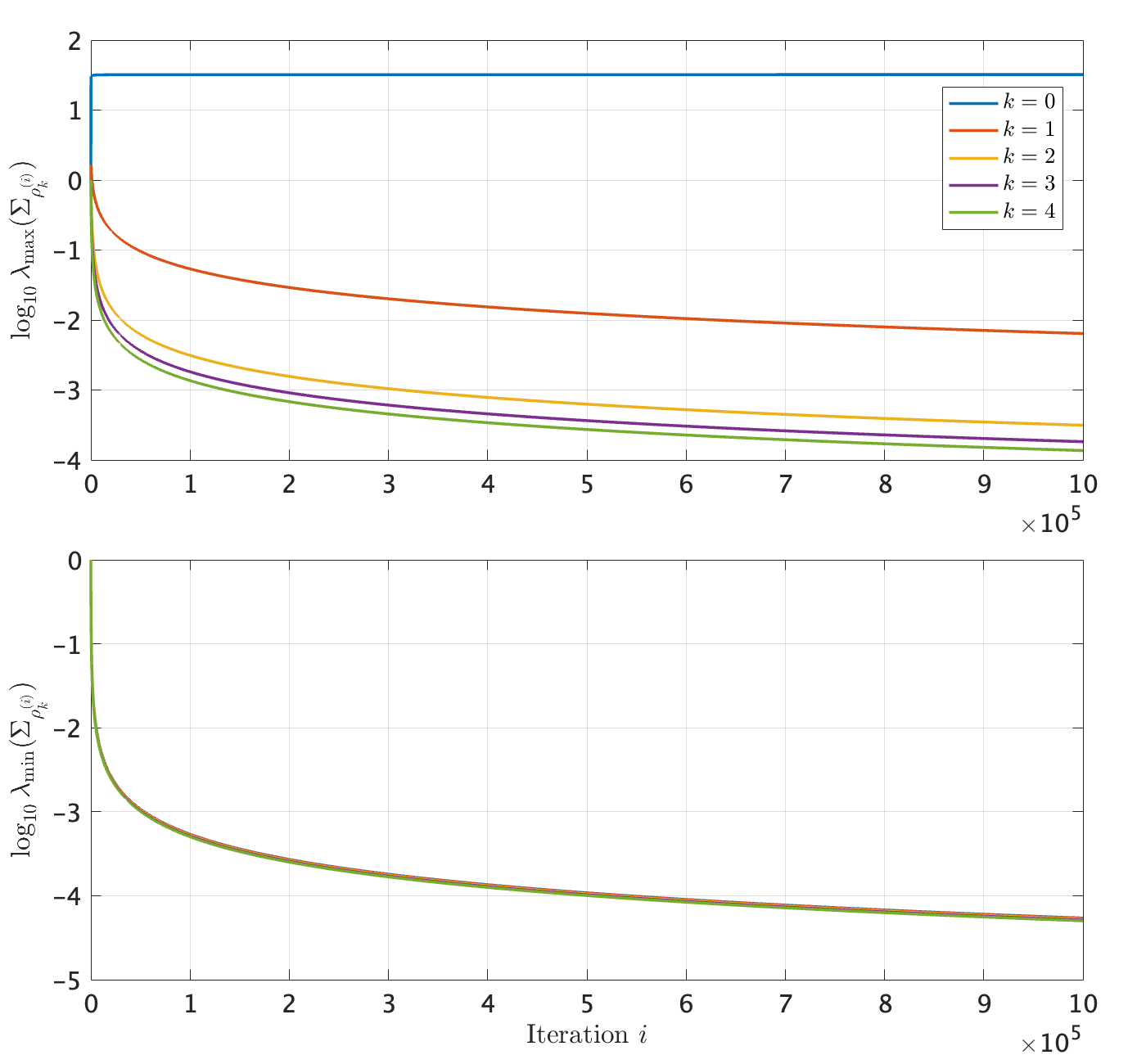}
        \subcaption{$\varepsilon = 10^{2}$}
        \label{fig:eps10p2}
      \end{minipage}
      &
      \begin{minipage}[t]{0.49\textwidth}
        \centering
        \includegraphics[width=0.95\linewidth]{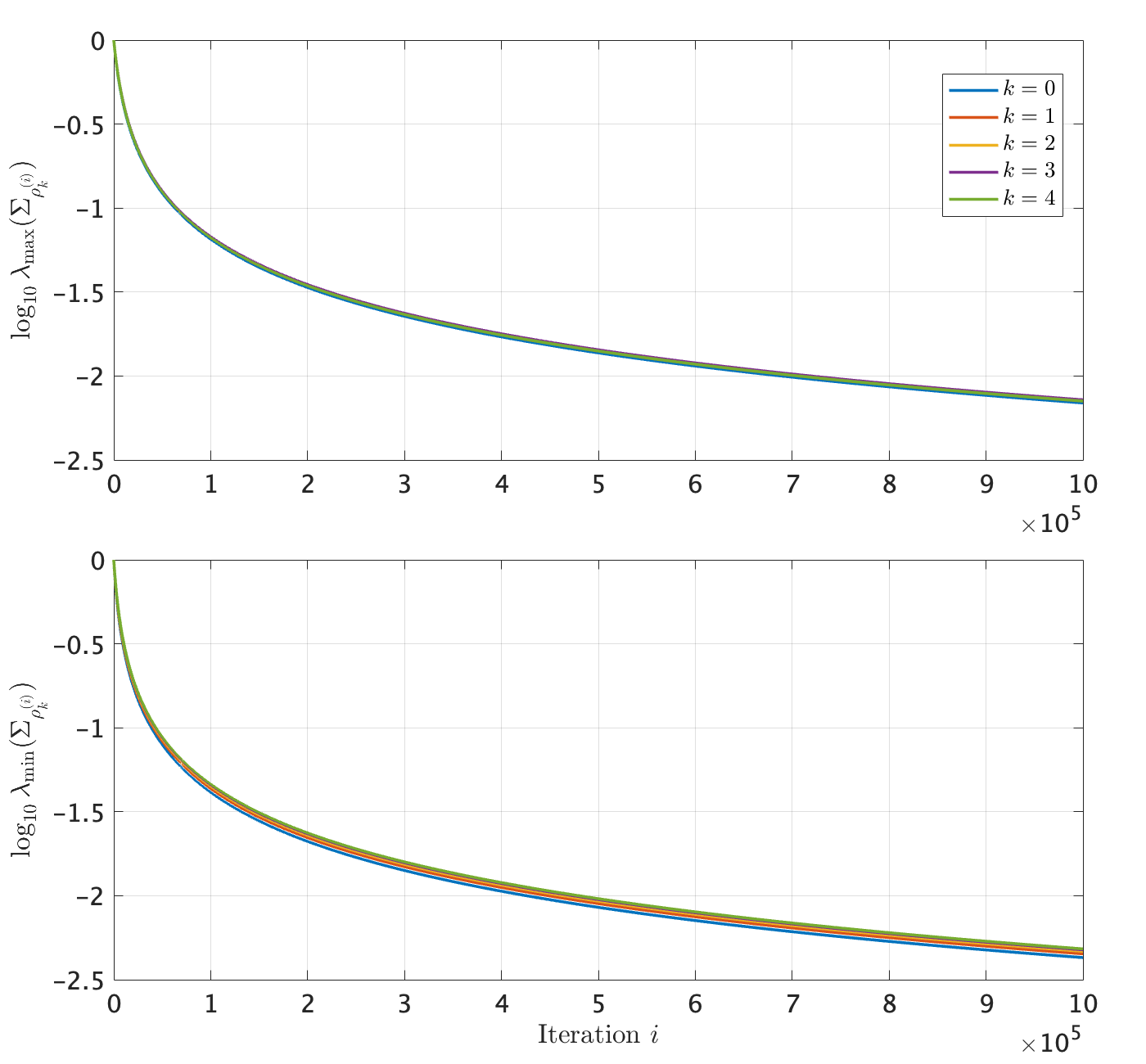}
        \subcaption{$\varepsilon = 10^{4}$}
        \label{fig:eps10p4}
      \end{minipage}

    \end{tabular}

    \caption{The trajectories of the maximum and minimum eigenvalues of $\Sigma_{\rho_{0}^{(i)}},\ldots,\Sigma_{\rho_{4}^{(i)}}$ for Problem \ref{prob:MIOCP} with $T=5$ and $\varepsilon = 10^{-6}, 1, 10^{2},$ and $10^{4}$.}
    \label{fig:convergence of prior covariance}
\end{figure*}

Fig. \ref{fig:convergence of prior covariance} shows the trajectories of the maximum and minimum eigenvalues of $\Sigma_{\rho_{0}^{(i)}},\ldots, \Sigma_{\rho_{4}^{(i)}}$ for different $\varepsilon$.
Note that $\varepsilon = 10^{-6}$ and $\varepsilon = 10^{4}$ satisfy the assumptions of Theorems \ref{thm:condition where algorithm converges to nonzero covariance matrix} and \ref{thm:condition where algorithm converges to zero covariance matrix}, respectively, and $\varepsilon = 1, 10^{2}$ do not satisfy these assumptions.

As shown in Figs. \ref{fig:eps10m6} and \ref{fig:eps10p4}, the covariance matrices exhibit the behaviors predicted by Theorems \ref{thm:condition where algorithm converges to nonzero covariance matrix} and \ref{thm:condition where algorithm converges to zero covariance matrix}, respectively.
For $\varepsilon=10^{-6}$, the maximum eigenvalues converge to positive values at all time steps, showing that the limiting prior covariances are nonzero.
In contrast, for $\varepsilon=10^{4}$, both the maximum and minimum eigenvalues converge toward zero, which consistently implies convergence to the zero covariance matrix.

Figs. \ref{fig:eps1} and \ref{fig:eps10p2} illustrate the behavior in the intermediate regime, where neither of the sufficient conditions in Theorems \ref{thm:condition where algorithm converges to nonzero covariance matrix} and \ref{thm:condition where algorithm converges to zero covariance matrix} is satisfied.
In this regime, the minimum eigenvalue can converge toward zero while the maximum eigenvalue remains positive.
Thus, the limiting prior covariance can be singular but nonzero, indicating that stochasticity is retained only in a subspace of the input space.
This observation is consistent with the discussion following \eqref{eq:equation that equilibrium points satisfy}.
In addition, for $\varepsilon=10^{2}$, the prior covariance converges to the zero matrix at more time steps than for $\varepsilon=1$.
Thus, at least in this numerical example, a larger $\varepsilon$ within the intermediate regime tends to produce zero covariance at more time steps.

\section{Conclusion}\label{sec:Conclusion}
In this paper, we investigated the MIOCP for stochastic discrete-time linear systems with quadratic costs and Gaussian priors.
We first reviewed the alternating optimization algorithm for the MIOCP and then analyzed fundamental properties of its optimal solutions.
In particular, we established the existence of an optimal solution and derived a characterization of the optimal prior covariance.
We further derived directly verifiable sufficient conditions on the temperature parameter for a nondegenerate stochastic feedback policy and a deterministic open-loop policy.
We also established convergence of the alternating optimization algorithm to the set of its fixed points and derived sufficient conditions under which the limiting policy is a stochastic feedback one or a deterministic open-loop one.
Numerical experiments justified these theoretical results of the alternating algorithm.

Future work includes extending the present analysis to a model-free RL setting, where the system model is unknown, and studying how learning from sampled data affects policy properties.
Another research direction is mutual information optimal density control, where both the initial and terminal state distributions are prescribed.
In stochastic control, the relation with Schr\"{o}dinger bridges \cite{peyre2019computational} has been studied \cite{beghi2002relative, chen2016robust, chen2016relation}.
The relationship between mutual information optimal density control and Schr\"{o}dinger bridges is of interest.

\begin{ack}                               
This work was supported by JSPS KAKENHI Grant Number 21H04875.  
\end{ack}

\appendix
\section*{Appendix}

\section{Proof of Proposition \ref{prop:optimal policy of MIOCP for fixed prior}} \label{app:Proof of Proposition of optimal policy of MIOCP for fixed prior}

Define the value function associated with Problem \ref{prob:MIOCP} with $\rho$ fixed as
\begin{align}
    V(k,x) :=& \min_{\pi_{k}}\mathbb{E}\left[ \frac{1}{2}\|u_{k}\|_{R_{k}}^{2} + \frac{1}{2}\|x_{k}\|_{Q_{k}}^{2} \right. \nonumber\\
    &+ \varepsilon \mathcal{D}_{\text{KL}}\left[\pi_{k}(\cdot|x)\| \rho_{k}\right]\nonumber\\
    &+\left.\mathbb{E}[ V(k+1, A_{k}x + B_{k}u_{k}+w_{k})] \mid x_{k} = x\right],\nonumber \\
    &x \in \mathbb{R}^{n}, k \in \llbracket 0, T-1 \rrbracket, \label{eq:definition of value function for k < T}\\
    V(T,x) := &\frac{1}{2}\|x\|_{F}^{2}, x\in \mathbb{R}^{n}. \label{eq:definition of value function for k = T}
\end{align}
In addition, define the corresponding Q-function as
\begin{align*}
    &\mathcal{Q}_{k}(x,u) := \frac{1}{2}\|u\|_{R_{k}}^{2} + \frac{1}{2}\|x\|_{Q_{k}}^{2} \\
    &\hspace{51pt}+ \mathbb{E}[V(k+1, A_{k}x + B_{k}u+w_{k})],\\
    &x\in \mathbb{R}^{n}, u \in \mathbb{R}^{m}, k \in \llbracket 0, T-1 \rrbracket.
\end{align*}
Noting that the KL divergence term implicitly requires $\rho_{k}\gg \pi_{k}(\cdot|x)$, we have
\begin{align*}
    &\mathbb{E}\left[ \frac{1}{2}\|u_{k}\|_{R_{k}}^{2} + \frac{1}{2}\|x\|_{Q_{k}}^{2} + \varepsilon \mathcal{D}_{\text{KL}}\left[\pi_{k}(\cdot|x)\| \rho_{k}\right] \right.\nonumber\\
    &+\left. \mathbb{E}[V(k+1, A_{k}x + B_{k}u_{k}+w_{k})] \mid x_{k} = x\right]\\
    &= \varepsilon\int_{\mathbb{R}^{m}} \left\{\log{\frac{d\pi_{k}}{d\rho_{k}}}(u|x) + \frac{1}{\varepsilon}\mathcal{Q}_{k}(x,u) \right\}d\pi_{k}(u|x)\\
    &=\varepsilon \mathcal{D}_{\text{KL}}\left[\pi_{k}(\cdot|x)\|\frac{\rho_{k,\mathcal{Q}_{k}}(\cdot|x)}{z_{k}}\right]-\varepsilon \log{z_{k}},
\end{align*}
where $\rho_{k,\mathcal{Q}_{k}}$ is defined as 
\begin{align*}
    \rho_{k,\mathcal{Q}_{k}}(\chi|x)=\int_{\chi}\exp \left(-\frac{1}{\varepsilon}\mathcal{Q}_{k}(x,u)\right)d\rho_{k}(u)\ \forall \chi \in \mathcal{B}_{m}
\end{align*}
and $z_{k}:= \int_{\mathbb{R}^{m}} d\rho_{k,\mathcal{Q}_{k}}(u|x)$ is a normalization constant.
Therefore, the optimal policy satisfies $\pi_{k}^{\rho}(\cdot|x)=\rho_{k,\mathcal{Q}_{k}}(\cdot|x)/z_{k}$.

To derive the characteristic function of $\pi_{T-1}^{\rho}(\cdot|x)$, let us calculate $\mathcal{Q}_{T-1}(x,u)$.
\begin{align*}
    &\hspace{-10pt}\mathcal{Q}_{T-1}(x,u)\\
    =& \frac{1}{2}\|u\|_{\varepsilon C_{T-1}}^{2}+ x^{\top}A_{T-1}^{\top}\Pi_{T}B_{T-1}u \\
    &+\frac{1}{2}x^{\top}\left(A_{T-1}^{\top}\Pi_{T} A_{T-1} + Q_{T-1} \right)x  + \frac{1}{2}\mathrm{Tr}[\Pi_{T}\Sigma_{w_{T-1}}].
\end{align*}
Hence, the characteristic function of $\pi_{T-1}^{\rho}(\cdot|x)$ satisfies
\begin{align}
    &\hspace{-10pt}\int_{\mathbb{R}^{m}}\exp({\mathrm{i}s^{\top}u})d\pi_{T-1}^{\rho}(u|x)\nonumber\\
    \propto &  \int_{\mathbb{R}^{m}}\exp\left( \left(\mathrm{i}s - \frac{1}{\varepsilon}B_{T-1}^{\top}\Pi_{T}A_{T-1}x\right)^{\top}u\right. \nonumber\\
    &\left.-\frac{1}{2}\|u\|_{C_{T-1}}^{2}\right)d\rho_{T-1}(u). \label{eq:characteristic function of optimal policy}
\end{align}
Suppose that $\Sigma_{\rho_{T-1}}\neq 0$.
Let us choose a full column rank matrix $\bar{\Sigma}_{\rho_{T-1}} \in \mathbb{R}^{m\times \text{rank}(\Sigma_{\rho_{T-1}})}$ that satisfies
\begin{align}
    \Sigma_{\rho_{T-1}}=\bar{\Sigma}_{\rho_{T-1}}\bar{\Sigma}_{\rho_{T-1}}^{\top}. \label{eq:decomposition of positive semidefinite matrix}
\end{align}
By using $\bar{\Sigma}_{\rho_{T-1}}$, the random variable $u\sim \rho_{T-1}$ can be rewritten as $u= \bar{\Sigma}_{\rho_{T-1}}v, v\sim \mathcal{N}(0,I)$.
Then, \eqref{eq:characteristic function of optimal policy} can be calculated as
\begin{align*}
     &\int_{\mathbb{R}^{d}}\exp\left( \left(\mathrm{i}s - \frac{1}{\varepsilon}B_{T-1}^{\top}\Pi_{T}A_{T-1}x\right)^{\top}\bar{\Sigma}_{\rho_{T-1}}v\right. \\
     &\left.  -\frac{1}{2}\|\bar{\Sigma}_{\rho_{T-1}}v\|_{C_{T-1}}^{2}\right) \tilde{\mathcal{N}}(v|0,I)dv \\
     &\propto  \exp \left( \mathrm{i}s^{\top}\mu_{\pi_{T-1}^{\rho}} - \frac{1}{2}\|s\|_{\Sigma_{\pi_{T-1}^{\rho}}}^{2}  \right).
\end{align*}
Because the characteristic function of a Gaussian distribution $\mathcal{N}(\mu,\Sigma)$ is given by $\exp(\mathrm{i}s^{\top}\mu-\frac{1}{2}\|s\|_{\Sigma}^{2})$ \cite{feller1991introduction}, this result implies that \eqref{eq:optimal policy of MIOCP for fixed prior} holds for $k = T-1$ if $\Sigma_{\rho_{T-1}}\neq 0$.
Next, we suppose that $\Sigma_{\rho_{T-1}}=0$.
Then, \eqref{eq:characteristic function of optimal policy} is proportional to $\exp(\mathrm{i}s^{\top}0)$, which implies that $\pi_{T-1}^{\rho}(\cdot |x)=\mathcal{N}(0,0)$.
This result coincides with \eqref{eq:optimal policy of MIOCP for fixed prior} because $\mu_{\pi_{T-1}^{\rho}}=0$ and $\Sigma_{\pi_{T-1}^{\rho}}=0$ when $\Sigma_{\rho_{T-1}}=0$.
Therefore, \eqref{eq:optimal policy of MIOCP for fixed prior} holds for $k = T-1$.
For simplicity of notation, we formally define
\begin{align}
    \bar{\Sigma}_{\rho_{T-1}}=0\in \mathbb{R}^{m\times m} \label{eq:decomposition of positive semidefinite matrix whtn it is zero matrix}
\end{align}
if $\Sigma_{\rho_{T-1}}=0$ henceforth.

The value function for $k = T-1$ can be rewritten as
\begin{align}
    &\hspace{-10pt}V(T-1, x)=  -\varepsilon  \log{z_{T-1}} \nonumber\\
    = &\frac{1}{2}x^{\top}\left(A_{T-1}^{\top}\Pi_{T} A_{T-1} + Q_{T-1} \right)x + \frac{1}{2}\mathrm{Tr}[\Pi_{T}\Sigma_{w_{T-1}}]\nonumber \\
    &- \varepsilon  \log \left\{ \int_{\mathbb{R}^{m}} \exp \left( -\frac{1}{\varepsilon}x^{\top}A_{T-1}^{\top}\Pi_{T}B_{T-1}u \right. \right. \nonumber \\
    &-\left. \left.\frac{1}{2}\|u\|_{C_{T-1}}^{2}\right)d\rho_{T-1}(u)\right\}. \label{eq:value function for k = T-1}
\end{align}
If $\Sigma_{\rho_{T-1}} \neq 0$, by following the same way used to rewrite \eqref{eq:characteristic function of optimal policy}, the argument of the logarithm of the last term in \eqref{eq:value function for k = T-1} can be calculated as
\begin{align*}
     &\hspace{-10pt}  \int_{\mathbb{R}^{m}} \exp \left( -\frac{1}{\varepsilon}x^{\top}A_{T-1}^{\top}\Pi_{T}B_{T-1}u-\frac{1}{2}\|u\|_{C_{T-1}}^{2}\right) \\
     &\hspace{-10pt}\times d\rho_{T-1}(u)\\
    =&\frac{1}{\sqrt{|I+\bar{\Sigma}_{\rho_{T-1}}^{\top}C_{T-1}\bar{\Sigma}_{\rho_{T-1}}|}}\\
     &\times\exp \left(   \frac{1}{2}\left\|\frac{1}{\varepsilon}B_{T-1}^{\top}\Pi_{T}A_{T-1}x\right\|_{\Sigma_{\pi_{T-1}}}^{2}\right).
\end{align*}
This result also covers the case where $\Sigma_{\rho_{T-1}}=0$.
By using this result, \eqref{eq:value function for k = T-1} can be rewritten as
\begin{align}
    &\hspace{-10pt}V(T-1, x) \nonumber\\
    = &\frac{1}{2}\|x\|_{\Pi_{T-1}}^{2} + \frac{1}{2}\mathrm{Tr}[\Pi_{T}\Sigma_{w_{T-1}}] \nonumber\\
    &+\frac{\varepsilon }{2}\log |I+\bar{\Sigma}_{\rho_{T-1}}^{\top}C_{T-1}\bar{\Sigma}_{\rho_{T-1}}|.\label{eq:arranged value function for k = T-1}
\end{align}
Since the first term of the right-hand side of \eqref{eq:arranged value function for k = T-1} takes the same form as $V(T, x)$ and the other terms are independent of $x$, we can derive the policy \eqref{eq:optimal policy of MIOCP for fixed prior} for $k = T-2, T-3, \ldots, 0,$ recursively by following the same procedure as for $k = T-1$, which completes the proof.

\section{Proof of Proposition \ref{prop:optimal prior for fixed policy}}\label{app:Proof of Proposition of optimal prior for fixed policy}

Since $\pi$ is fixed, the state and input quadratic costs are independent of $\rho$, and thus we have
\begin{align*}
    &\min_{\rho} \mathbb{E}\left[ \sum_{k=0}^{T-1} \left\{\frac{1}{2}\|u_{k}\|_{R_{k}}^{2} + \varepsilon \mathcal{D}_{\text{KL}}[\pi_{k}(\cdot|x_{k}) \| \rho_{k}] \right\}\right. \\
    &+\left. \frac{1}{2}\|x_{T}\|_{F}^{2} \right]\\
    &\Leftrightarrow  \min_{\rho_{k}} \mathbb{E}\left[\mathcal{D}_{\text{KL}}[\pi_{k}(\cdot|x_{k}) \| \rho_{k}]\right], k \in \llbracket 0, T-1 \rrbracket.
\end{align*}
Let $\eta_{k}$ denote the marginal distribution of $u_{k}$ induced by $\pi_{k}$ and the distribution $p$ of $x_{k}$, that is,
\begin{align*}
    \eta_{k}(\chi):= \int_{\mathbb{R}^{n}} \pi_{k}(\chi|x_{k})dp(x_{k}), \chi\in\mathcal{B}_{m}.
\end{align*}
Since $\pi \in \mathcal{P}$ and $p(x_{k})=\mathcal{N}(x_{k}|0,\Sigma_{x_{k}})$, we have
\begin{align*}
    \eta_{k} = \mathcal{N} \left( 0,\Sigma_{\pi_{k}}+P_{k}\Sigma_{x_{k}}P_{k}^{\top} \right).
\end{align*}
Moreover, since $\mathrm{Im}(P_k) \subseteq \mathrm{Im}(\Sigma_{\pi_{k}})$, it follows that $\mathrm{Im} \left( \Sigma_{\pi_{k}}+P_{k}\Sigma_{x_{k}}P_{k}^{\top} \right) = \mathrm{Im}(\Sigma_{\pi_{k}})$, and hence $\pi_{k}(\cdot|x)$ and $\eta_{k}$ have the same support for every $x$.
Thus, $\pi_{k}(\cdot|x) \ll \eta_{k}$.
In addition, if $\mathbb{E}[ \mathcal{D}_{\mathrm{KL}} [\pi_{k}(\cdot|x_{k})\|\rho_{k}]] <\infty$, then $\pi_{k}(\cdot|x)\ll\rho_{k}$ for $p(x_{k})$-almost every $x$, which implies $\eta_{k}\ll\rho_{k}$.
Then, whenever the expected KL divergence is finite, we have
\begin{align}
    &\mathbb{E}\left[\mathcal{D}_{\mathrm{KL}}[\pi_k(\cdot|x_k)\|\rho_k]\right] \nonumber\\
    &= \int_{\mathbb{R}^{n}}\int_{\mathbb{R}^{m}}\log\frac{d\pi_{k}(\cdot|x_{k})}{d\rho_{k}}(u)d\pi_{k}(u|x_{k})dp(x_{k})\nonumber\\
    &= \int_{\mathbb{R}^{n}}\int_{\mathbb{R}^{m}}\log\frac{d\pi_{k}(\cdot|x_{k})}{d\eta_{k}}(u)d\pi_{k}(u|x_{k})dp(x_{k}) \nonumber\\
    &\quad+\int_{\mathbb{R}^{n}}\int_{\mathbb{R}^{m}}\log\frac{d\eta_{k}}{d\rho_{k}}(u)d\pi_{k}(u|x_{k})dp(x_{k})\nonumber\\
    &=\mathcal{I}(x_{k};u_{k})+\mathcal{D}_{\mathrm{KL}}[\eta_{k}\|\rho_{k}]. \label{eq:KL decomposition prior}
\end{align}

The distribution $\eta_{k}$ belongs to $\mathcal{R}$.
In addition,
\begin{align*}
    \mathcal{D}_{\mathrm{KL}}[\eta_{k}\|\rho_{k}]\geq 0,
\end{align*}
where equality holds if and only if $\rho_{k}=\eta_{k}$.
Since the mutual information $\mathcal{I}(x_k;u_k)$ in \eqref{eq:KL decomposition prior} is independent of $\rho_{k}$, it follows that
\begin{align*}
    \rho_{k}^{\pi} = \eta_{k} = \mathcal{N}\left(0,\Sigma_{\pi_{k}}+P_{k}\Sigma_{x_{k}}P_{k}^{\top}\right),
\end{align*}
which completes the proof.

\section{Justification of Zero Mean Priors} \label{app:Justification of Zero Mean Priors}
    In this appendix, we show the following proposition.
    \begin{prop}
        For Problem \ref{prob:MIOCP} with the nonzero prior class $\tilde{\mathcal{R}}$, the prior class $\tilde{\mathcal{R}}$ can be restricted to $\mathcal{R}$ without loss of optimality in the sense that the infimum of Problem \ref{prob:MIOCP} does not change.
        Moreover, if an optimal solution $(\pi^{*},\rho^{*})$ exists, then $\mu_{\rho_{k}^{*}}=0, k\in\llbracket0,T-1\rrbracket$.
        \qedtheorem
        \label{prop:zero mean prior}
    \end{prop}

    Proposition \ref{prop:zero mean prior} justifies the statement in Remark \ref{rem:nonzero mean prior}.
    In the rest of this appendix, we show the proof of Proposition \ref{prop:zero mean prior}.

    Consider an arbitrary feasible pair $(\pi,\rho)$ of Problem \ref{prob:MIOCP} with $\tilde{\mathcal{R}}$ having a finite objective value.
    Since $R_{k}\succ0$, the finiteness of the objective implies that $\mathbb{E}[\|u_{k}\|^{2}]<\infty$.
    Then, define $\mu_{u_{k}}:=\mathbb{E}[u_{k}]$.
    
    We first show that replacing the mean of $\rho_{k}$ with $\mu_{u_{k}}$ does not increase the objective.
    Suppose first that $\Sigma_{\rho_{k}}\neq0$.
    The finiteness of $\mathbb{E}\left[\mathcal{D}_{\mathrm{KL}}[\pi_k(\cdot|x_k)\|\rho_k]\right]$ implies that $\pi_{k}(\cdot|x)\ll\rho_{k}$ for $p(x_{k})$-almost every $x$.
    Consequently, the marginal distribution of $u_{k}$ is also absolutely continuous with respect to $\rho_{k}$ and is therefore supported on $\mu_{\rho_k}+\mathrm{Im}(\Sigma_{\rho_k})$.
    It follows that
    \begin{align}
        \mu_{u_{k}}-\mu_{\rho_{k}} \in \mathrm{Im}(\Sigma_{\rho_{k}}). \label{eq:mean difference in covariance image}
    \end{align}
    Now define
    \begin{align*}
        \hat{\rho}_{k}:=\mathcal{N}(\mu_{u_{k}},\Sigma_{\rho_{k}}).
    \end{align*}
    By \eqref{eq:mean difference in covariance image}, $\rho_{k}$ and $\hat{\rho}_{k}$ have the same support.
    On this common support,
    \begin{align}
        \log\frac{d\hat{\rho}_{k}}{d\rho_{k}}(u) = \frac{1}{2} \|u-\mu_{\rho_{k}}\|_{\Sigma_{\rho_{k}}^{\dagger}}^{2} - \frac{1}{2} \|u-\mu_{u_{k}}\|_{\Sigma_{\rho_{k}}^{\dagger}}^{2}. \label{eq:likelihood ratio shifted prior}
    \end{align}
    Therefore,
    \begin{align}
        &\mathbb{E}\left[\mathcal{D}_{\mathrm{KL}}[\pi_{k}(\cdot|x_{k})\|\rho_{k}]\right]-\mathbb{E}\left[\mathcal{D}_{\mathrm{KL}}[\pi_{k}(\cdot|x_{k})\|\hat{\rho}_{k}]\right]\nonumber\\
        &=\frac{1}{2}\mathbb{E}\left[\|u_{k}-\mu_{\rho_{k}}\|_{\Sigma_{\rho_{k}}^{\dagger}}^{2} - \|u_{k}-\mu_{u_{k}}\|_{\Sigma_{\rho_{k}}^{\dagger}}^{2}\right]\nonumber\\
        &=\frac{1}{2}\|\mu_{u_{k}}-\mu_{\rho_{k}}\|_{\Sigma_{\rho_{k}}^{\dagger}}^{2}\geq 0. \label{eq:KL reduction by prior mean}
    \end{align}
    If $\Sigma_{\rho_{k}}=0$, the finiteness of the KL divergence implies $u_{k}=\mu_{\rho_{k}}$ almost surely, and hence
    $\mu_{u_{k}}=\mu_{\rho_{k}}$.
    Since the prior does not affect the state dynamics or the quadratic costs, \eqref{eq:KL reduction by prior mean} yields
    \begin{align}
        J(\pi,\hat{\rho}) \leq J(\pi,\rho), \label{eq:prior mean replacement}
    \end{align}
    where $\hat{\rho}:=\{\hat{\rho}_{k}\}_{k=0}^{T-1}$.

    Next, we construct a feasible pair with a zero mean prior.
    Since the initial state and the disturbances are zero mean,
    \begin{align}
        \mu_{x_{0}}&=0, \nonumber\\
        \mu_{x_{k+1}} &= A_{k} \mu_{x_{k}}+ B_{k} \mu_{u_{k}}, k\in\llbracket0,T-1\rrbracket. \label{eq:mean dynamics centering}
    \end{align}
    Introduce the centered variables
    \begin{align*}
        \tilde{x}_{k}:= x_{k}-\mu_{x_{k}},
        \tilde{u}_{k}:= u_{k}-\mu_{u_{k}}.
    \end{align*}
    Then, \eqref{eq:mean dynamics centering} gives
    \begin{align}
        \tilde{x}_{k+1} = A_{k} \tilde{x}_{k} + B_{k}\tilde{u}_{k} + w_{k}. \label{eq:centered dynamics}
    \end{align}
    Define the translated policy by
    \begin{align}
        \tilde{\pi}_{k}(\chi|\tilde{x}):= \pi_{k}\left( \chi+\mu_{u_{k}}| \tilde{x}+\mu_{x_{k}}\right), \chi\in\mathcal{B}_{m}, \label{eq:translated policy}
    \end{align}
    and the translated prior by
    \begin{align}
        \tilde{\rho}_{k}:= \mathcal{N}(0,\Sigma_{\rho_{k}}). \label{eq:translated prior}
    \end{align}
    The invariance of the KL divergence under a common translation gives
    \begin{align}
        \mathcal{D}_{\mathrm{KL}}[\tilde{\pi}_{k}(\cdot|\tilde{x})\| \tilde{\rho}_{k}]= \mathcal{D}_{\mathrm{KL}} [\pi_{k}(\cdot|\tilde{x}+\mu_{x_{k}}) \|\hat{\rho}_{k}]. \label{eq:KL invariance under centering}
    \end{align}
    On the other hand, the quadratic costs are decomposed as
    \begin{align}
        \mathbb{E}\left[\|x_{k}\|_{Q_{k}}^{2}\right]&=\mathbb{E}\left[\|\tilde{x}_{k}\|_{Q_{k}}^{2}\right]+\|\mu_{x_{k}}\|_{Q_{k}}^{2}, \label{eq:state cost centering}\\
        \mathbb{E}\left[\|u_{k}\|_{R_{k}}^{2}\right]&=\mathbb{E}\left[\|\tilde{u}_{k}\|_{R_{k}}^{2}\right]+\|\mu_{u_{k}}\|_{R_{k}}^{2},\label{eq:input cost centering}\\
        \mathbb{E}\left[\|x_{T}\|_{F}^{2}\right]&=\mathbb{E}\left[\|\tilde{x}_{T}\|_{F}^{2}\right]+\|\mu_{x_{T}}\|_{F}^{2}. \label{eq:terminal cost centering}
    \end{align}
    Combining
    \eqref{eq:KL invariance under centering}--\eqref{eq:terminal cost centering}, we obtain
    \begin{align}
        &J(\pi,\hat{\rho})-J(\tilde{\pi},\tilde{\rho})\nonumber\\
        &=\frac{1}{2}\sum_{k=0}^{T-1}\left(\|\mu_{x_{k}}\|_{Q_{k}}^{2}+\|\mu_{u_{k}}\|_{R_{k}}^{2}\right)+\frac{1}{2}\|\mu_{x_{T}}\|_{F}^{2} \geq 0. \label{eq:objective reduction by centering}
    \end{align}

    Since $\tilde{\rho}\in\mathcal{R}$, \eqref{eq:prior mean replacement} and \eqref{eq:objective reduction by centering} show that, for every finite-cost feasible pair $(\pi,\rho)$ of Problem \ref{prob:MIOCP} with $\tilde{\mathcal{R}}$, there exists a feasible pair $(\tilde{\pi},\tilde{\rho})$ with $\tilde{\rho}\in\mathcal{R}$ such that
    \begin{align*}
        J(\tilde{\pi},\tilde{\rho}) \leq J(\pi,\rho).
    \end{align*}
    Hence,
    \begin{align*}
        \inf_{\pi,\rho\in\mathcal{R}}J(\pi,\rho) \leq \inf_{\pi,\rho\in\tilde{\mathcal{R}}}J(\pi,\rho).
    \end{align*}
    The reverse inequality follows immediately from $\mathcal{R}\subset \tilde{\mathcal{R}}$.
    Therefore, the first claim of Proposition \ref{prop:zero mean prior} holds.
    
    Finally, suppose that an optimal solution $(\pi^{*},\rho^{*})$ exists.
    Optimality, \eqref{eq:KL reduction by prior mean}, and \eqref{eq:objective reduction by centering} imply
    \begin{align*}
        \sum_{k=0}^{T-1}\left( \|\mu_{x_{k}^{*}}\|_{Q_{k}}^{2} + \|\mu_{\rho_{k}^{*}}\|_{R_{k}}^{2}\right) + \|\mu_{x_{T}^{*}}\|_{F}^{2} =0,
    \end{align*}
    where $\mu_{x_{k}^{*}}$ is the solution to \eqref{eq:mean dynamics centering} with $\mu_{u_{k}}=\mu_{\rho_{k}^{*}}$.
    Since $R_k\succ0$, it follows that $\mu_{\rho_{k}^{*}}=0, k\in \llbracket 0,T-1\rrbracket$, which completes the proof.

\section{Proof of Lemma \ref{lem:inequalities for condition where optimal covariance matrices are full-rank}} \label{app:Proof of Lemma of inequalites}

First, we prove the inequalities in \eqref{eq:inequality of Pi}.
For $X\succ 0$ and $Y\succeq 0$, we have
\begin{align}
    &Y^{1/2}(Y^{1/2}X^{-1}Y^{1/2}+I)^{-1}Y^{1/2}\nonumber\\
    &=Y^{1/2}(I-Y^{1/2}(X+Y)^{-1}Y^{1/2})Y^{1/2}\nonumber\\
    &=Y-Y(X+Y)^{-1}Y = X(X+Y)^{-1}Y\nonumber\\
    &=X-X(X+Y)^{-1}X \preceq X. \label{eq:matrix inequality for proof of Lemma 2}
\end{align}
By applying \eqref{eq:matrix inequality for proof of Lemma 2} with $X=\varepsilon\left(R_{k}+B_{k}^{\top}\Pi_{k+1}B_{k}\right)^{-1}, Y=\Sigma_{\rho_{k}}$, we obtain
\begin{align}
    \Pi_{k}
    =& Q_{k}+A_{k}^{\top}\Pi_{k+1}A_{k}-\frac{1}{\varepsilon}A_k^\top\Pi_{k+1}B_k\Sigma_{\rho_k}^{1/2}\nonumber\\
    &\times \left(I+\Sigma_{\rho_k}^{1/2}C_k\Sigma_{\rho_k}^{1/2}\right)^{-1}\Sigma_{\rho_k}^{1/2}B_k^\top\Pi_{k+1}A_k\nonumber\\
    \succeq&Q_k+A_k^\top\Pi_{k+1}A_k\nonumber\\
    &-
    A_k^\top\Pi_{k+1}B_k
    \left(
    R_k+B_k^\top\Pi_{k+1}B_k
    \right)^{-1}
    B_k^\top\Pi_{k+1}A_k.
    \label{eq:lower bound of Pi in proof}
\end{align}
For $Y\succeq0$, define
\begin{align*}
    f_k(Y):=&Y-YB_k\left(R_k+B_k^\top YB_k\right)^{-1}B_k^\top Y\\
    ={}&Y^{1/2}\left(I+Y^{1/2}B_kR_k^{-1}B_k^\top Y^{1/2}\right)^{-1}Y^{1/2}.
\end{align*}
Thus, $f_k(Y)\succeq0$ for any $Y\succeq0$.
In particular, since $Q_k\succeq0$ and $F\succeq0$, the recursion
\begin{align*}
    \check{\Pi}_k =Q_k+A_k^\top f_k(\check{\Pi}_{k+1})A_k, \check{\Pi}_T=F,
\end{align*}
implies
\begin{align}
    \check{\Pi}_k\succeq0, k\in\llbracket0,T\rrbracket. \label{eq:check Pi positive semidefinite}
\end{align}
In addition, $f_k$ is monotone on the positive semidefinite cone.
Indeed, for $Y_1\succeq Y_2\succ0$,
\begin{align*}
    f_k(Y_1) &= \left(Y_1^{-1}+B_kR_k^{-1}B_k^\top\right)^{-1}\\
    &\succeq\left(Y_2^{-1}+B_kR_k^{-1}B_k^\top\right)^{-1} =f_k(Y_2).
\end{align*}
For $Y_1\succeq Y_2\succeq0$, the same inequality follows by applying the above result to $Y_1+\delta I\succeq Y_2+\delta I\succ0$
and letting $\delta\downarrow0$.
Now, suppose that $\Pi_{k+1}\succeq\check{\Pi}_{k+1}$.
From \eqref{eq:lower bound of Pi in proof} and the monotonicity of $f_k$, we have
\begin{align*}
    \Pi_k &\succeq Q_k+A_k^\top f_k(\Pi_{k+1})A_k\\
    &\succeq Q_k+A_k^\top f_k(\check{\Pi}_{k+1})A_k = \check{\Pi}_k.
\end{align*}
Since $\Pi_T=\check{\Pi}_T=F$, backward induction yields
\begin{align}
    \Pi_k \succeq \check{\Pi}_k \succeq0, k\in\llbracket0,T\rrbracket. \label{eq:lower bound Pi final}
\end{align}
Since the last term in the Riccati equation \eqref{eq:Riccati difference equation} is negative semidefinite, we have
\begin{align}
    \Pi_k \preceq Q_k+A_k^\top\Pi_{k+1}A_k. \label{eq:upper bound Pi one step}
\end{align}
Suppose that $\hat{\Pi}_{k+1} \succeq \Pi_{k+1}.$
Then, from \eqref{eq:upper bound Pi one step},
\begin{align*}
    \Pi_k &\preceq Q_k+A_k^\top\Pi_{k+1}A_k\\
    &\preceq Q_k+A_k^\top\hat{\Pi}_{k+1}A_k = \hat{\Pi}_k.
\end{align*}
Since $\Pi_T=\hat{\Pi}_T=F$, backward induction gives
\begin{align}
    \hat{\Pi}_k \succeq \Pi_k,k\in\llbracket0,T\rrbracket. \label{eq:upper bound Pi final}
\end{align}
Combining \eqref{eq:lower bound Pi final} and \eqref{eq:upper bound Pi final} proves \eqref{eq:inequality of Pi}.

Finally, from $\hat{\Pi}_{k+1} \succeq \Pi_{k+1} \succeq \check{\Pi}_{k+1} \succeq0$, we obtain
\begin{align*}
    R_k+B_k^\top\hat{\Pi}_{k+1}B_k &\succeq R_k+B_k^\top\Pi_{k+1}B_k\\
    &\succeq R_k+B_k^\top\check{\Pi}_{k+1}B_k \succ0.
\end{align*}
It hence follows that
\begin{align*}
    &\varepsilon \left(R_k+B_k^\top\check{\Pi}_{k+1}B_k\right)^{-1} \succeq \varepsilon \left( R_k+B_k^\top\Pi_{k+1}B_k \right)^{-1} \nonumber\\
    &\succeq \varepsilon \left( R_k+B_k^\top\hat{\Pi}_{k+1}B_k \right)^{-1} \succ0.
\end{align*}
Therefore, $\hat{\Gamma}_k \succeq \Gamma_k \succeq \check{\Gamma}_k \succ0,k\in\llbracket0,T-1\rrbracket$, which proves \eqref{eq:inequality of Gamma} and completes the proof.

\section{Proof of Lemma \ref{lem:J hat is continuous}}\label{app:Proof of Lemma of continuity of J hat}

We show that $\check{J}$ is jointly continuous in $(\Sigma_{\rho_{0}},\ldots,\Sigma_{\rho_{T-1}})\in\mathcal{M}_{T}$.
Using \eqref{eq:Gamma}, the Riccati recursion \eqref{eq:Riccati difference equation} can be rewritten as
\begin{align}
\Pi_{k} ={}&Q_{k}+A_{k}^{\top}\Pi_{k+1}A_{k} -\frac{1}{\varepsilon}A_{k}^{\top}\Pi_{k+1}B_{k} \nonumber \\
&\times \left\{ \Gamma_{k} -\Gamma_{k}(\Gamma_{k}+\Sigma_{\rho_{k}})^{-1}\Gamma_{k} \right\} B_{k}^{\top}\Pi_{k+1}A_{k}. \label{eq:Riccati for continuity}
\end{align}
Since $\Pi_{T}=F$ is constant, $\Gamma_{T-1}$ is constant and positive definite.
Hence, \eqref{eq:Riccati for continuity} shows that $\Pi_{T-1}$ is continuous in $\Sigma_{\rho_{T-1}}$.

Suppose that $\Pi_{k+1}$ is jointly continuous in $(\Sigma_{\rho_{k+1}},$ $\ldots,\Sigma_{\rho_{T-1}})$.
Then, $\Gamma_{k}=\varepsilon(R_k+B_k^\top\Pi_{k+1}B_k)^{-1}$ is also jointly continuous in $(\Sigma_{\rho_{k+1}},\ldots,\Sigma_{\rho_{T-1}})$.
In addition, $\Gamma_{k}\succ0$ by Lemma \ref{lem:inequalities for condition where optimal covariance matrices are full-rank}, and hence $\Gamma_{k}+\Sigma_{\rho_{k}}\succ0$ for any $\Sigma_{\rho_{k}}\succeq0$.
Hence, \eqref{eq:Riccati for continuity} implies that $\Pi_{k}$ is jointly continuous in $(\Sigma_{\rho_{k}},\ldots,\Sigma_{\rho_{T-1}})$.
By backward induction, $\Pi_{k}$ and $\Gamma_{k}$ are jointly continuous functions of the corresponding prior covariance matrices for any $k\in\llbracket0,T-1\rrbracket$.

Finally, from \eqref{eq:objective function of simplified MIOCP},
\begin{align*}
2\check{J} = \sum_{k=0}^{T} \mathrm{Tr}[\Pi_{k}\Sigma_{w_{k-1}}] +\sum_{k=0}^{T-1}\varepsilon\log\frac{|\Sigma_{\rho_{k}}+\Gamma_{k}|}{|\Gamma_{k}|} .
\end{align*}
Since $\Gamma_{k}\succ0$ and $\Sigma_{\rho_{k}}+\Gamma_{k}\succ0$, all the terms on the right-hand side are continuous on $\mathcal{M}^{T}$, which completes the proof.

\section{Proof of Lemma \ref{lem:J hat is coercive}}\label{app:Proof of Lemma that J hat is coercive}

Since $\Pi_{k}\succeq 0$, $\Sigma_{x_{\mathrm{ini}}}\succeq 0$, and $\Sigma_{w_{k}}\succeq 0$, all the trace terms in $\check{J}$ are
nonnegative.
In addition,
\begin{align*}
    \log\frac{|\Sigma_{\rho_{k}}+\Gamma_{k}|}{|\Gamma_{k}|} = \log \left| I+\Gamma_{k}^{-1/2} \Sigma_{\rho_{k}} \Gamma_{k}^{-1/2}\right| \geq 0.
\end{align*}
Therefore, for any $k\in\llbracket0,T-1\rrbracket$,
\begin{align}
    2\check{J} \geq \varepsilon \log \left| I+\Gamma_{k}^{-1/2} \Sigma_{\rho_{k}} \Gamma_{k}^{-1/2}\right|. \label{eq:lower bound J coercivity}
\end{align}
For any positive semidefinite matrix $X$ of size $m$, we have
\begin{align*}
    &|I+X| = \prod_{i=1}^{m} (1+\lambda_{i}(X)) \geq 1 + \sum_{i=1}^{m}\lambda_{i}(X)\\
    &= 1+\mathrm{Tr}[X],
\end{align*}
where $\lambda_{1}(X),\ldots,\lambda_{m}(X)$ is the eigenvalues of $X$.
Applying this inequality to $X=\Gamma_k^{-1/2}\Sigma_{\rho_k}\Gamma_k^{-1/2}$ yields
\begin{align}
    &\left|I+\Gamma_{k}^{-1/2}\Sigma_{\rho_{k}}\Gamma_{k}^{-1/2}\right| \geq 1+\mathrm{Tr}[\Gamma_{k}^{-1}\Sigma_{\rho_{k}}].\label{eq:det lower bound coercivity}
\end{align}
From \eqref{eq:inequality of Gamma}, $\Gamma_{k}^{-1} \succeq \hat{\Gamma}_{k}^{-1}$.
Hence,
\begin{align}
    \mathrm{Tr}[\Gamma_{k}^{-1}\Sigma_{\rho_{k}}] &\geq \mathrm{Tr}[\hat{\Gamma}_{k}^{-1}\Sigma_{\rho_{k}}] \geq \lambda_{\min}(\hat{\Gamma}_{k}^{-1})\mathrm{Tr}[\Sigma_{\rho_{k}}]\nonumber\\
    &=\frac{\mathrm{Tr}[\Sigma_{\rho_{k}}]}{\lambda_{\max}(\hat{\Gamma}_{k})} \geq \frac{\|\Sigma_{\rho_{k}}\|}{\lambda_{\max}(\hat{\Gamma}_{k})}, \label{eq:trace lower bound coercivity}
\end{align}
where the last inequality follows from $\mathrm{Tr}[\Sigma_{\rho_{k}}]\geq\|\Sigma_{\rho_{k}}\|$ for $\Sigma_{\rho_{k}}\succeq 0$.

Combining
\eqref{eq:lower bound J coercivity}--\eqref{eq:trace lower bound coercivity}, we obtain
\begin{align}
    2\check{J} \geq \varepsilon \log \left( 1+ \frac{\|\Sigma_{\rho_{k}}\|}{\lambda_{\max}(\hat{\Gamma}_{k})}
    \right) \label{eq:J coercivity final bound}
\end{align}
for every $k\in\llbracket0,T-1\rrbracket$.
Since $\hat{\Gamma}_k$ is independent of $\{\Sigma_{\rho_{k}}\}_{k=0}^{T-1}$, the right-hand side of \eqref{eq:J coercivity final bound} tends to infinity as $\|\Sigma_{\rho_{k}}\|\rightarrow \infty$.
Therefore, if there exists at least one $k\in\llbracket0,T-1\rrbracket$ such that $\|\Sigma_{\rho_{k}}\|\to\infty$, then $\check{J}\rightarrow\infty$.
This proves the coercivity of $\check{J}$ on $\mathcal{M}_{T}$.

\section{Proof of Lemma \ref{lem:derivative of J_check}} \label{app:Proof of Lemma of derivative of J_check}

    We start by deriving the derivative of $\check{J}$.
    We first calculate the derivative of $\check{J}$ with respect to $\Sigma_{\rho_{0}}$.
    Because $\Gamma_{k},\ldots \Gamma_{{T-1}}, \Pi_{k+1},\ldots , \Pi_{T}$ are independent of $\Sigma_{\rho_{0}},\ldots,\Sigma_{\rho_{k}}$, we have
    \begin{align}
        2\frac{\partial\check{J}}{\partial \Sigma_{\rho_{0}}} = \frac{\partial}{\partial \Sigma_{\rho_{0}}}\mathrm{Tr}\left[\Pi_{0} \Sigma_{w_{-1}}\right] + \varepsilon  \frac{\partial}{\partial \Sigma_{\rho_{0}}} \log |\Sigma_{\rho_{0}} + \Gamma_{{0}}|. \label{eq:derivative of J for l=0}
    \end{align}
    From \eqref{eq:Riccati difference equation} and \eqref{eq:Gamma}, $\Pi_{k}$ can be rewritten as
    \begin{align*}
        \Pi_{k} = & Q_{k}+ A_{k}^{\top} \Pi_{k+1} A_{k}- \frac{1}{\varepsilon }A_{k}^{\top} \Pi_{k+1} B_{k} \\
        &\times \{ \Gamma_{k}-\Gamma_{k}(\Gamma_{k} + \Sigma_{\rho_{k}})^{-1}\Gamma_{k} \}B_{k}^{\top} \Pi_{k+1} A_{k}.
    \end{align*}
    By using formulas of matrix calculus \cite{petersen2008matrix}, the first and second terms in the right-hand side of \eqref{eq:derivative of J for l=0} can be calculated as follows, respectively.
    \begin{align*}
        &\hspace{-10pt}\frac{\partial}{\partial \Sigma_{\rho_{0}}}\mathrm{Tr}\left[\Pi_{0} \Sigma_{w_{-1}}\right]\\
        =& \frac{1}{\varepsilon }\frac{\partial}{\partial \Sigma_{\rho_{0}}}\mathrm{Tr}\left[ A_{0}^{\top} \Pi_{1} B_{0} \Gamma_{{0}}(\Gamma_{{0}} + \Sigma_{\rho_{0}})^{-1}\right. \\
        &\times \left.\Gamma_{{0}} B_{0}^{\top} \Pi_{1} A_{0} \Sigma_{w_{-1}}\right]\\
        =& - \varepsilon  \left[ (\Gamma_{{0}} + \Sigma_{\rho_{0}})^{-1} \frac{\Gamma_{{0}}}{\varepsilon }B_{0}^{\top} \Pi_{1} A_{0} \Sigma_{w_{-1}} \right. \\
        &\times \left.A_{0}^{\top} \Pi_{1} B_{0} \frac{\Gamma_{{0}}}{\varepsilon } (\Gamma_{{0}} + \Sigma_{\rho_{0}})^{-1}  \right],\\
        &\varepsilon  \frac{\partial}{\partial \Sigma_{\rho_{0}}} \log |\Sigma_{\rho_{0}} + \Gamma_{{0}}|
        = \varepsilon  (\Gamma_{{0}} + \Sigma_{\rho_{0}})^{-1}.
    \end{align*}
    By substituting \eqref{eq:E}, \eqref{eq:L}, and $\Sigma_{w_{-1}}=\Sigma_{x_{0}}$, it follows that
    \begin{align}
        \frac{2}{\varepsilon }\frac{\partial\check{J}}{\partial \Sigma_{\rho_{0}}} =&  L_{0} \left( \Gamma_{{0}} + \Sigma_{\rho_{0}} - E_{0}\Sigma_{x_{0}} E_{0}^{\top} \right) L_{0}.  \nonumber
    \end{align}
    Next, we consider the case $k\in \llbracket 1, T-1 \rrbracket$.
    As in the case $k = 0$, the derivative of $\check{J}$ with respect to $\Sigma_{\rho_{k}}$ can be arranged as follows:
    \begin{align}
        &\hspace{-10pt}2\frac{\partial\check{J}}{\partial \Sigma_{\rho_{k}}}\nonumber\\
        = & \frac{\partial}{\partial \Sigma_{\rho_{k}}}\mathrm{Tr}\left[\Pi_{0} \Sigma_{w_{-1}}\right] \nonumber\\
        &+ \frac{\partial}{\partial \Sigma_{\rho_{k}}}\sum_{l=0}^{k-1}\left(\varepsilon \log \frac{ |\Sigma_{\rho_{l}} + \Gamma_{{l}}|}{|\Gamma_{{l}}|} + \mathrm{Tr}[\Pi_{l+1}\Sigma_{w_{l}}] \right)\nonumber \\
        &+ \varepsilon  \frac{\partial}{\partial \Sigma_{\rho_{k}}}\log |\Sigma_{\rho_{k}} + \Gamma_{k}|.\nonumber
    \end{align}
    From a straightforward calculation, for the differential $d\Sigma_{\rho_{k}}$, we have
    \begin{align*}
        &\hspace{-10pt}\mathrm{Tr}\left[(d\Pi_{0}) \Sigma_{w_{-1}}\right] + \varepsilon  d\left( \frac{\log |\Sigma_{\rho_{0}} + \Gamma_{{0}}|}{|\Gamma_{{0}}|}\right) \\
        &\hspace{-10pt}+ \mathrm{Tr}[(d\Pi_{1})\Sigma_{w_{0}}]\\
        =& \mathrm{Tr}\left[(d\Pi_{1})\left(A_{0}-\frac{1}{\varepsilon}B_{0}\Sigma_{\pi_{0}^{\rho}}B_{0}^{\top}\Pi_{1}A_{0}\right)\Sigma_{w_{-1}} \right. \\
        &\times \left.\left(A_{0}-\frac{1}{\varepsilon}B_{0}\Sigma_{\pi_{0}^{\rho}}B_{0}^{\top}\Pi_{1}A_{0}\right)^{\top}\right] \\
        &+ \mathrm{Tr}[(d\Pi_{1})B_{0}\Sigma_{\pi_{0}^{\rho}}B_{0}^{\top}]  + \mathrm{Tr}[(d\Pi_{1})\Sigma_{w_{0}}]\\
        =&\mathrm{Tr}\left[(d\Pi_{1}) \Sigma_{x_{1}}\right].
    \end{align*}
    By applying this result recursively, it follows that
    \begin{align*}
         2\frac{\partial \check{J}}{\partial \Sigma_{\rho_{k}}} = & \frac{\partial}{\partial \Sigma_{\rho_{k}}}\mathrm{Tr}\left[\Pi_{k} \Sigma_{x_{k}}\right] + \varepsilon  \frac{\partial}{\partial \Sigma_{\rho_{k}}}\log |\Sigma_{\rho_{k}} + \Gamma_{k}|,
    \end{align*}
    where $\Sigma_{x_{k}}$ in the first term can be regarded as a constant with respect to $\frac{\partial}{\partial \Sigma_{\rho_{k}}}$.
    Then, applying the same argument for $k=0$, it follows that
    \begin{align}
        \frac{2}{\varepsilon }\frac{\partial\check{J}}{\partial \Sigma_{\rho_{k}}} =&  L_{k} \left( \Gamma_{k} + \Sigma_{\rho_{k}} - E_{k}\Sigma_{x_{k}} E_{k}^{\top} \right) L_{k}. \nonumber
    \end{align}

    Now, we derive \eqref{eq:directional derivative of J_check}.
    Since $\Gamma_{k}\succ 0$, we have $\Gamma_{k}+\Sigma_{\rho_{k}}\succ 0$ for any $\Sigma_{\rho_{k}}\succeq 0$.
    Hence, the above differential calculation is valid on $\mathcal{M}_{T}$, including its boundary.
    Since $\mathcal{M}_{T}$ is convex, $(\bar{\Sigma}_{\rho_{0}}+t(S_{0}-\bar{\Sigma}_{\rho_{0}}),\ldots,\bar{\Sigma}_{\rho_{T-1}}+t(S_{T-1}-\bar{\Sigma}_{\rho_{T-1}}) \in\mathcal{M}_{T}$ for any $t\in[0,1]$.
    Using the derivative obtained above, we have
    \begin{align*}
        &\check{J}\left(\bar{\Sigma}_{\rho_{0}}+t(S_{0}-\bar{\Sigma}_{\rho_{0}}),\ldots,\bar{\Sigma}_{\rho_{T-1}}+t(S_{T-1}-\bar{\Sigma}_{\rho_{T-1}})\right)\\
        &-\check{J}\left(\bar{\Sigma}_{\rho_0},\ldots,\bar{\Sigma}_{\rho_{T-1}}\right)\\
        &=t\sum_{k=0}^{T-1}\mathrm{Tr}\left[\check{J}^{\prime}_{k}\left(\bar{\Sigma}_{\rho_0},\ldots,\bar{\Sigma}_{\rho_{T-1}}\right)(S_k-\bar{\Sigma}_{\rho_k})\right]+o(t),
    \end{align*}
    which implies \eqref{eq:directional derivative of J_check}.

\section{Proof of Theorem \ref{thm:characterization of optimal solution}} \label{app:Proof of Theorem of characterization of optimal solution}
    In this proof, denote $G_{k}:=\check{J}_{k}^{\prime}(\Sigma_{\rho_{0}^{*}},\ldots, \Sigma_{\rho_{T-1}^{*}})$.
    From \cite[Proposition 2.1.1]{borwein2006convex}, a necessary condition for $\{\Sigma_{\rho_{k}^{*}}\}_{k=0}^{T-1}$ to be an optimal solution is that
    \begin{align}
        \sum_{k=0}^{T-1}\mathrm{Tr}\left[ G_{k}(S_{k}-\Sigma_{\rho_{k}^{*}})\right] \geq 0 \ \forall (S_{0},\ldots , S_{T-1})\in \mathcal{M}_{T}. \nonumber 
    \end{align}
    Since $\mathcal{M}_{T}$ is the Cartesian product of $\mathbb{S}_{\succeq0}^{m}$, this condition is equivalent to
    \begin{align}
        &\mathrm{Tr}\left[ G_{k}(S_{k}-\Sigma_{\rho_{k}^{*}})\right] \geq 0 \label{eq:first order optimality} 
    \end{align}
    for any $S_{k}\in \mathbb{S}_{\succeq 0}^{m}, k \in \llbracket0,T-1 \rrbracket$.
    By choosing $S_{k}=\Sigma_{\rho_k}^*+vv^{\top}$ for arbitrary $v\in\mathbb{R}^{m}$, we obtain $G_{k} \succeq 0$.
    Choosing $S_{k}=0$ in \eqref{eq:first order optimality} and using $G_{k}\succeq 0$ and $\Sigma_{\rho_{k}}^{*}\succeq0$ gives $\mathrm{Tr}(G_{k}\Sigma_{\rho_{k}}^{*})=0$, and thus
    \begin{align}
        G_{k}\Sigma_{\rho_{k}^{*}}=0. \label{eq:complementarity condition}
    \end{align}
    Define $Y_{k}^{*}:= \Gamma_{k}^{-1/2} \Sigma_{\rho_{k}^{*}} \Gamma_{k}^{-1/2}$.
    Using Lemma \ref{lem:derivative of J_check}, we have
    \begin{align*}
        G_{k} = &\ \frac{\epsilon}{2} \Gamma_{k}^{-1/2}(I+Y_{k}^{*})^{-1}(I+Y_{k}^{*}-\Xi_{k}^{*})\\
        &\times(I+Y_{k}^{*})^{-1}\Gamma_{k}^{-1/2}.
    \end{align*}
    Hence, $G_{k}\succeq0$ implies
    \begin{align}
        I+Y_{k}^{*}-\Xi_{k}^{*}\succeq0. \label{eq:normalized psd}
    \end{align}
    In addition, \eqref{eq:complementarity condition} and the fact that $(I+Y_{k}^{*})^{-1}$ commutes with $Y_{k}^{*}$ yield
    \begin{align}
        (I+Y_{k}^{*}-\Xi_{k}^{*})Y_{k}^{*}=0. \label{eq:normalized_comp}
    \end{align}
    From \eqref{eq:normalized psd} and \eqref{eq:normalized_comp}, it follows that
    \begin{align*}
        Y_{k}^{*}\succeq 0, Y_{k}^{*}-H_{k}\succeq 0, (Y_{k}^{*}-H_{k})Y_{k}^{*}=0,
    \end{align*}
    where $H_{k}:=\Xi_{k}^{*}-I$.
    Since $Y_{k}^{*}$ is symmetric, the last equality implies that $H_{k}$ and $Y_{k}^{*}$ commute and can be simultaneously diagonalized.
    Then, the conditions above reduce to
    \begin{align*}
        y_{j}^{*} \geq 0, y_{j}^{*}-h_{j} \geq 0,(y_{j}^{*}-h_{j})y_{j}^{*} =0
    \end{align*}
    for each $j$, where $y_{j}^{*}$ and $h_{j}$ are the eigenvalues of $Y_{k}^{*}$ and $H_{k}$.
    If $h_{j}>0$, the second condition implies $y_{j}^{*}>0$, and the third condition therefore gives $y_{j}^{*}=h_{j}$.
    If $h_{j}\leq 0$, the third condition yields $y_{j}^{*}=0$.
    Therefore, $y_{j}^{*}=\max(h_{j},0)$, which implies $Y_{k}^{*}=(H_{k})_{+}=(\Xi_{k}^{*}-I)_{+}$.
    By recalling $Y_{k}^{*}= \Gamma_{k}^{-1/2} \Sigma_{\rho_{k}^{*}} \Gamma_{k}^{-1/2}$, the claim of Theorem \ref{thm:characterization of optimal solution} holds.

\section{Proof of Theorem \ref{thm:condition where optimal covariance matrices are full-rank}} \label{app:Proof of Theorem of condition where optimal covariance matrices are full-rank}

    From Lemma \ref{lem:inequalities for condition where optimal covariance matrices are full-rank}, we have $\hat{\Pi}_{k+1}\succeq\Pi_{k+1}\succeq\check{\Pi}_{k+1}$.
    Recall that $E_{k} = (R_{k}+B_{k}^{\top}\Pi_{k+1}B_{k})^{-1} B_{k}^{\top}\Pi_{k+1}A_{k}$.
    Then,
    \begin{align*}
        \sigma_{\min}(E_{k})&\geq \frac{\sigma_{\min}(A_{k})\sigma_{\min}(B_{k}^{\top}\Pi_{k+1})}{\lambda_{\max}(R_{k}+B_{k}^{\top}\hat{\Pi}_{k+1}B_{k})}.
    \end{align*}
    Because we have
    \begin{align*}
        \lambda_{\min}(B_{k}^{\top}\check{\Pi}_{k+1}B_{k}) &\leq v^{\top} B_{k}^{\top}\Pi_{k+1}B_{k}v\\
        &\leq \sigma_{\max}(B_{k}) \|\Pi_{k+1}B_{k}v\|
    \end{align*}
    for any $v\in\mathbb{R}^m$ with $\|v\|=1$, it follows that
    \begin{align}
        \sigma_{\min}(E_{k}) \geq \frac{ \sigma_{\min}(A_{k}) \lambda_{\min} (B_{k}^{\top}\check{\Pi}_{k+1}B_{k})}{\sigma_{\max}(B_{k}) \lambda_{\max} (R_{k}+B_{k}^{\top}\hat{\Pi}_{k+1}B_{k})} =:\underline{e}_{k}. \label{eq:lower bound of E}
    \end{align}
    From the assumptions of $B_{k}^{\top}\check{\Pi}_{k+1}B_{k}\succ 0$ and the invertibility of $A_{k}$, we have $\underline{e}_{k}>0$, and thus $E_{k}$ is full row rank and $E_{k}E_{k}^{\top} \succeq \underline{e}_{k}^{2}I$.
    Since $\Sigma_{x_{k}}\succeq\Sigma_{w_{k-1}}\succ0$, we have
    \begin{align}
        E_{k}\Sigma_{x_{k}}E_{k}^{\top} &\succeq \lambda_{\min}(\Sigma_{w_{k-1}}) E_{k}E_{k}^{\top} \nonumber\\
        &\succeq \lambda_{\min}(\Sigma_{w_{k-1}}) \underline{e}_{k}^{2} I =\alpha_{k} I. \label{eq:lower bound of E Sigma E}
    \end{align}
    In addition, Lemma \ref{lem:inequalities for condition where optimal covariance matrices are full-rank} gives
    \begin{align*}
        \Gamma_{k} \preceq \hat{\Gamma}_{k} = \varepsilon (R_{k}+B_{k}^{\top}\check{\Pi}_{k+1}B_{k})^{-1}.
    \end{align*}
    Therefore,
    \begin{align}
        \Gamma_{k}-E_{k}\Sigma_{x_{k}}E_{k}^{\top} &\preceq \varepsilon (R_{k}+B_{k}^{\top}\check{\Pi}_{k+1}B_{k})^{-1} -\alpha_{k} I \nonumber\\
        &= -\check{M}_{k} \prec0. \label{eq:negative bound for theorem 1}
    \end{align}

    Recall that a necessary condition for $\{\Sigma_{\rho_{k}^{*}}\}_{k=0}^{T-1}$ to be an optimal solution is that
    \begin{align}
        &\mathrm{Tr}\left[ \check{J}_{k}^{\prime}(\Sigma_{\rho_{0}^{*}},\ldots, \Sigma_{\rho_{T-1}^{*}})(S_{k}-\Sigma_{\rho_{k}^{*}})\right] \geq 0 \nonumber
    \end{align}
    for any $S_{k}\in \mathbb{S}_{\succeq 0}^{m}, k \in \llbracket0,T-1 \rrbracket$ from \eqref{eq:first order optimality}.

    Now, we show that $\Sigma_{\rho_{k}^{*}}\succ 0$.
    By \eqref{eq:derivative of J_check}, it follows that
    \begin{align*}
        &\mathrm{Tr}\left[ \check{J}_{k}^{\prime}(\Sigma_{\rho_{0}^{*}},\ldots, \Sigma_{\rho_{T-1}^{*}})(S_{k}-\Sigma_{\rho_{k}^{*}})\right]\\
       &=\frac{\varepsilon}{2}\mathrm{Tr}\left[L_{k}(\Sigma_{\rho_{k}^{*}}+\Gamma_{k}-E_{k}\Sigma_{x_{k}}E_{k}^{\top})L_{k}(S_{k}-\Sigma_{\rho_{k}^{*}})\right].
    \end{align*}
    Because we choose $\varepsilon$ such that $\check{M}_{k} \succ 0$, we have
    \begin{align*}
        \Gamma_{k}-E_{k}\Sigma_{x_{k}}E_{k}^{\top} \preceq -\check{M}_{k} \prec 0.
    \end{align*}
    Suppose that $\Sigma_{\rho_{k}^{*}}$ has at least one zero eigenvalue.
    Then, $L_{k}(\Sigma_{\rho_{k}^{*}}+\Gamma_{k}-E_{k}\Sigma_{x_{k}}E_{k}^{\top})L_{k}$ has at least one negative eigenvalue.
    Let $U_{k}\text{diag}(\sigma_{k,1},\ldots, \sigma_{k,m})U_{k}^{\top}$ be the eigenvalue decomposition of $L_{k}(\Sigma_{\rho_{k}^{*}}+\Gamma_{k}-E_{k}\Sigma_{x_{k}}E_{k}^{\top})L_{k}$, where $\text{diag}(\sigma_{k,1},\ldots, \sigma_{k,m})$ is the diagonal matrix with entries $\sigma_{k,1},\ldots, \sigma_{k,m}$ on the diagonal and $\sigma_{k,m}$ is a negative eigenvalue.
    If we choose $S_{k} = \Sigma_{\rho_{k}^{*}} + U_{k}\text{diag}(0,\ldots ,0,1)U_{k}^{\top}$, it follows that
    \begin{align*}
        &\hspace{-10pt}\mathrm{Tr}\left[\check{J}_{k}^{\prime}(\Sigma_{\rho_{0}^{*}},\ldots , \Sigma_{\rho_{T-1}^{*}})(S_{k}-\Sigma_{\rho_{k}^{*}}) \right]\\
        =&\frac{\varepsilon}{2}\mathrm{Tr}\left[U_{k}\text{diag}(\sigma_{k,1},\ldots, \sigma_{k,m})U_{k}^{\top} \right.\\
        &\times \left.U_{k}\text{diag}(0,\ldots ,0,1)U_{k}^{\top}\right]\\
        =&\frac{\varepsilon}{2}\mathrm{Tr}\left[U_{k}\text{diag}(0,\ldots ,0,\sigma_{k,m})U_{k}^{\top}\right]<0.
    \end{align*}
    This contradicts the fact that $\Sigma_{\rho_{k}^{*}}$ is an optimal solution, which completes the proof.

\section{Proof of Theorem \ref{thm:condition where optimal covariance matrices are 0}} \label{app:Proof of Theorem of condition where optimal covariance matrices are 0}
    We will employ a similar argument to that used in the proof of Theorem \ref{thm:condition where optimal covariance matrices are full-rank}.
    We first show that $\Sigma_{\rho_0^*}=0$.
    The definition of $\bar{e}_{0}$ gives $\sigma_{\max}(E_{0}) \leq \bar{e}_{0}$.
    Hence,
    \begin{align*}
        E_{0}\Sigma_{x_{0}}^{\mathrm{zero}}E_{0}^{\top} &\preceq \lambda_{\max}(\Sigma_{x_{0}}^{\mathrm{zero}})\sigma_{\max}^2(E_{0})I \preceq \beta_{0}I.
    \end{align*}
    In addition, Lemma \ref{lem:inequalities for condition where optimal covariance matrices are full-rank} gives $\Gamma_0 \succeq \varepsilon (R_{0}+B_{0}^{\top}\hat{\Pi}_{1}B_{0})^{-1}$.
    Therefore,
    \begin{align}
        \Gamma_{0}-E_{0}\Sigma_{x_{0}}^{\mathrm{zero}}E_{0}^{\top} \succeq -\hat{M}_{0} \succ0. \label{eq:positive bound theorem 2}
    \end{align}
    From \eqref{eq:derivative of J_check}, we have
     \begin{align*}
        &\mathrm{Tr}\left[\check{J}_{0}^{\prime}(\Sigma_{\rho_{0}^{*}},\ldots , \Sigma_{\rho_{T-1}^{*}})(S_{0}-\Sigma_{\rho_{0}^{*}}) \right]\\
        &=\frac{\varepsilon}{2}\mathrm{Tr}\left[L_{0}(\Sigma_{\rho_{0}^{*}}+\Gamma_{{0}}-E_{0}\Sigma_{x_{0}}^{\mathrm{zero}}E_{0}^{\top})L_{0}(S_{0}-\Sigma_{\rho_{0}^{*}})\right],
    \end{align*}
    where $S_{0} \in \mathbb{S}_{\succeq 0}^{m}$.
    Suppose that $\Sigma_{\rho_{0}^{*}}\neq 0$.
    By choosing $S_{0}=0$, we have
    \begin{align*}
        &\frac{\varepsilon}{2}\mathrm{Tr}\left[L_{0}(\Sigma_{\rho_{0}^{*}}+\Gamma_{{0}}-E_{0}\Sigma_{x_{0}}^{\mathrm{zero}}E_{0}^{\top})L_{0}(S_{0}-\Sigma_{\rho_{0}^{*}})\right]\\
        &=-\frac{\varepsilon}{2}\mathrm{Tr}\left[\Sigma_{\rho_{0}^{*}}^{\frac{1}{2}}L_{0}(\Sigma_{\rho_{0}^{*}}+\Gamma_{{0}}-E_{0}\Sigma_{x_{0}}^{\mathrm{zero}}E_{0}^{\top})L_{0}\Sigma_{\rho_{0}^{*}}^{\frac{1}{2}}\right]\\
        &\leq  -\frac{\varepsilon}{2}\mathrm{Tr}\left[\Sigma_{\rho_{0}^{*}}^{\frac{1}{2}}L_{0}(\Sigma_{\rho_{0}^{*}}-\hat{M}_{0})L_{0}\Sigma_{\rho_{0}^{*}}^{\frac{1}{2}}\right]<0.
    \end{align*}
    This contradicts the optimality of $\Sigma_{\rho_{0}^{*}}$.
    It hence follows that $\Sigma_{\rho_{0}^{*}} = 0$, that is, $\pi_{0}^{*}(\cdot|x)=\mathcal{N}(0,0)$.
    Under this optimal policy, we have $\Sigma_{x_{1}}=\Sigma_{x_{1}}^{\mathrm{zero}}$.
    By applying this argument recursively, we obtain the desired result.

\section{Proof of Corollary \ref{cor:mutual information at optimal solution}} \label{app:Proof of Corollary of mutual information at optimal solution}

If $Y_{k}^{*}=0$, then $u_{k}$ is independent of $x_{k}$ from  Proposition \ref{prop:optimal policy of MIOCP for fixed prior} and Theorem \ref{thm:characterization of optimal solution}, and hence $\mathcal{I}(x_{k};u_{k})=0$.
In the rest of this proof, suppose $\mathrm{rank}(Y_{k}^{*})>0$.
Define $v_{k}:=\Gamma_{k}^{-1/2}u_{k}$.
Because this transformation is invertible, $\mathcal{I}(x_{k};u_{k})=\mathcal{I}(x_{k};v_{k})$.
Noting $\rho_{k}$ is the marginal distribution of $u_{k}$ from Proposition \ref{prop:optimal prior for fixed policy}, the marginal covariance matrix of $v_{k}$ is $Y_{k}^{*}:=\Gamma_{k}^{-1/2}\Sigma_{\rho_{k}^{*}}\Gamma_{k}^{-1/2}=(\Xi_{k}^{*}-I)_{+}$ from Theorem \ref{thm:characterization of optimal solution}.
In addition, the conditional covariance matrix of $v_{k}$ given $x=x_{k}$ is $Y_{k}^{*}(I+Y_{k}^{*})^{-1}$ from \eqref{eq:covariance matrix of optimal policy of MIOCP for fixed prior}.
Let us write $Y_{k}^{*}$ as
\begin{align*}
    Y_{k}^{*}=\bar{U}\bar{Y}_{k}\bar{U}^{\top},
\end{align*}
where $\bar{Y}_{k}\succ 0$ is a diagonal matrix consisting of all positive eigenvalues of $Y_{k}^{*}$ and $\bar{U} \in \mathbb{R}^{m \times \mathrm{rank}(Y_{k}^{*})}$ has orthonormal columns.
Define $\bar{v}_{k}:=\bar{U}^{\top}v_{k}$.
Since $v_{k}\in \mathrm{Im}(Y_{k}^{*})$, $v_{k}$ and $\bar{v}_{k}$ are in one-to-one correspondence on the support, and thus we have $\mathcal{I}(x_{k};v_{k})=\mathcal{I}(x_{k};\bar{v}_{k})$.
The marginal and conditional covariance matrices of $\bar{v}_{k}$ are $\bar{Y}_{k}$ and $\bar{Y}_{k}(I+\bar{Y}_{k})^{-1}$, respectively.
Therefore,
\begin{align*}
    \mathcal{I}(x_{k};u_{k})=\frac{1}{2}\log \frac{|\bar{Y}_{k}|}{|\bar{Y}_{k}(I+\bar{Y}_{k})^{-1}|}=\frac{1}{2}\log |(I+\bar{Y}_{k})|.
\end{align*}
By combining this with $Y_{k}^{*}=(\Xi_{k}^{*}-I)_{+}$, the claim of Corollary \ref{cor:mutual information at optimal solution} holds.

\section{Proof of Proposition \ref{prop:algorithm converges to equilibrium point}} \label{app:Proof of Proposition of algorithm converges to equilibrium point}
    We start by showing that $\rho = \mathcal{A}(\rho) \Leftrightarrow J(\rho) = J(\mathcal{A}(\rho))$.
    It trivially holds that $\rho = \mathcal{A}(\rho) \Rightarrow J(\rho) = J(\mathcal{A}(\rho))$.
    To show the converse, let us suppose that $J(\rho)=J(\mathcal{A}(\rho))$.
    Because we minimize $J$ alternatively in Algorithm \ref{alg:alternating optimization algorithm for MIOCPs}, it follows that $J(\rho)=J(\pi^{\rho},\rho)\geq J(\pi^{\rho},\mathcal{A}(\rho))\geq J(\pi^{\mathcal{A}(\rho)},\mathcal{A}(\rho))=J(\mathcal{A}(\rho))$.
    It hence follows that $J(\pi^{\rho},\rho)=J(\pi^{\rho},\mathcal{A}(\rho))$.
    Because the optimal prior for the fixed policy $\pi^{\rho}$ is unique from Proposition \ref{prop:optimal prior for fixed policy}, we have $\rho = \mathcal{A}(\rho)$.

    Now, we show that $\mathcal{E}\subset\{ \rho \in \mathcal{R} | \rho = \mathcal{A}(\rho)\}$.
    Because $J(\rho^{(i)})\leq J(\rho^{(0)})$, $\{\Sigma_{\rho_k^{(i)}}\}_{k=0}^{T-1}$ belongs to the level set
    \begin{align*}
        &\left\{ \left\{\Sigma_{\rho_{k}}\right\}_{k=0}^{T-1} \in \mathcal{M}_{T} \mid\right. \\
        &\left.\  \check{J}(\Sigma_{\rho_{0}},\ldots , \Sigma_{\rho_{T-1}}) \leq \check{J}\left(\Sigma_{\rho_{0}^{(0)}},\ldots , \Sigma_{\rho_{T-1}^{(0)}}\right)\right\}
    \end{align*}
    for any $i\in\mathbb{Z}_{\geq0}$.
    This level set is closed by Lemma \ref{lem:J hat is continuous} and bounded by Lemma \ref{lem:J hat is coercive}, and is therefore compact.
    Thus, by identifying $\rho^{(i)}$ with $(\Sigma_{\rho_{0}^{(i)}},\ldots,\Sigma_{\rho_{T-1}^{(i)}})$, the sequence $\{\rho^{(i)}\}_{i\in\mathbb{Z}_{\geq0}}$ has at least one cluster point \cite[Theorem 17.4]{willard2012general}.
    Hence, $\mathcal{E}$ is nonempty.

    Because Algorithm \ref{alg:alternating optimization algorithm for MIOCPs} alternately minimizes $J$ with respect to the policy and the prior, $\{J(\rho^{(i)})\}_{i\in\mathbb{Z}_{\geq0}}$ is nonincreasing.
    Since $J(\rho)\geq0$ for any $\rho\in\mathcal{R}$, there exists $\alpha\geq0$ such that $\lim_{i\rightarrow \infty}J(\rho^{(i)})=\alpha$.
    Let $\rho^{(\infty)}\in\mathcal{E}$ and choose a subsequence $\{\rho^{(i_{j})}\}$ such that $\rho^{(i_{j})}\to\rho^{(\infty)}$.
    The map $\mathcal{A}$ is continuous from the continuity argument in the proof of Lemma \ref{lem:J hat is continuous} together with the state covariance recursion and Propositions \ref{prop:optimal policy of MIOCP for fixed prior} and \ref{prop:optimal prior for fixed policy}.
    It hence follows that $\rho^{(i_{j}+1)} =\mathcal{A}(\rho^{(i_{j})})\rightarrow \mathcal{A}(\rho^{(\infty)})$.
    By the continuity of $J$, we have $J(\rho^{(\infty)}) = J(\mathcal{A}(\rho^{(\infty)})) = \alpha$.
    Consequently, $\rho^{(\infty)}=\mathcal{A}(\rho^{(\infty)})$.
    Therefore, the claim of Proposition \ref{prop:algorithm converges to equilibrium point} holds.

\section{Proof of Theorem \ref{thm:condition where algorithm converges to nonzero covariance matrix}} \label{app:Proof of Theorem of condition where algorithm converges to nonzero covariance matrix}
    In this proof, we denote $\Sigma_{\rho_{k}^{(i)}}$ and $\Sigma_{\rho_{k}^{(i+1)}}$ by $\Sigma_{\rho_{k}}$ and $\Sigma_{\rho_{k}}^{+}$, respectively.
    Note that $\Sigma_{x_{k}}, \Pi_{k}, k\in \llbracket0,T \rrbracket$ and $\Gamma_{k}, k\in \llbracket0,T-1 \rrbracket$ are calculated by using $\{\Sigma_{\rho_{k}}^{(i)}\}_{k=0}^{T-1}$.

    From the proof of Theorem \ref{thm:condition where optimal covariance matrices are full-rank}, we have
    \begin{align}
        E_k\Sigma_{x_k}E_k^\top-\Gamma_k
        \succeq \check M_k\succ0
        \label{eq:lower bound for Theorem 3}
    \end{align}
    for any $k\in\llbracket0,T-1\rrbracket$ and any iteration.
    In addition, from \eqref{eq:covariance matrix of optimal policy of MIOCP for fixed prior} and \eqref{eq:optimal prior for fixed policy}, the update in Algorithm \ref{alg:alternating optimization algorithm for MIOCPs} satisfies
    \begin{align}
        \Sigma_{\rho_{k}}^{+}-\Sigma_{\rho_{k}} = \Sigma_{\rho_{k}}L_{k}\left(E_{k}\Sigma_{x_{k}}E_{k}^{\top} -\Sigma_{\rho_{k}}-\Gamma_{k}\right)L_{k}\Sigma_{\rho_{k}}. \label{eq:covariance update theorem 3}
    \end{align}
    Since Algorithm \ref{alg:alternating optimization algorithm for MIOCPs} is initialized with $\Sigma_{\rho_{k}}^{(0)}\succ0$, we have $\Sigma_{\rho_{k}}^{(i)}\succ0$ for any $i \in \mathbb{Z}_{\geq 0}$.

    First, suppose that $\Sigma_{\rho_{k}}\prec\check{M}_{k}$.
    It follows from \eqref{eq:lower bound for Theorem 3} and \eqref{eq:covariance update theorem 3} that
    \begin{align*}
        \Sigma_{\rho_{k}}^{+}-\Sigma_{\rho_{k}}\succeq\Sigma_{\rho_{k}}L_{k}\left(\check{M}_{k}-\Sigma_{\rho_{k}}\right)L_{k}\Sigma_{\rho_{k}}\succ0.
    \end{align*}
    Therefore,
    \begin{align}
        \lambda_{\max}(\Sigma_{\rho_{k}}^{+}) > \lambda_{\max}(\Sigma_{\rho_{k}}). \label{eq:first case theorem 3}
    \end{align}

    Next, suppose that $\Sigma_{\rho_{k}}\prec\check{M}_{k}$ does not hold.
    Then,
    \begin{align}
        \lambda_{\max}(\Sigma_{\rho_{k}}) \geq\gamma_{k}, \gamma_{k}:=\lambda_{\min}(\check{M}_{k})>0. \label{eq:gamma theorem 3}
    \end{align}
    Using \eqref{eq:lower bound for Theorem 3} in \eqref{eq:covariance update theorem 3}, we have
    \begin{align*}
        \Sigma_{\rho_{k}}^{+} &\succeq \Sigma_{\rho_{k}} + \Sigma_{\rho_{k}}L_{k} (\check{M}_{k}-\Sigma_{\rho_{k}})L_{k}\Sigma_{\rho_{k}}\\
        &=\Sigma_{\rho_{k}}L_{k}(\check{M}_{k}+\Gamma_{k})L_{k}\Sigma_{\rho_{k}}+\Sigma_{\rho_{k}}-\Sigma_{\rho_{k}}L_{k}\Sigma_{\rho_{k}}\\
        &\succeq \Sigma_{\rho_{k}}L_{k} (\check{M}_{k}+\Gamma_{k}) L_{k}\Sigma_{\rho_{k}},
    \end{align*}
    where the last inequality follows from
    \begin{align*}
        \Sigma_{\rho_{k}} - \Sigma_{\rho_{k}}L_{k}\Sigma_{\rho_{k}} =&\ \Sigma_{\rho_{k}} - \Sigma_{\rho_{k}}(\Gamma_{k}+\Sigma_{\rho_{k}})^{-1}\Sigma_{\rho_{k}}\\
        =&\ \Sigma_{\rho_{k}}^{1/2}(\Sigma_{\rho_{k}}^{1/2}\Gamma_{k}^{-1}\Sigma_{\rho_{k}}^{1/2}+I)^{-1}\Sigma_{\rho_{k}}^{1/2}
        \succeq0.
    \end{align*}
    From the proof of Proposition \ref{prop:algorithm converges to equilibrium point}, $\{\rho^{(i)}\}_{i\in\mathbb Z_{\geq 0}}$ is contained in a compact set.
    It hence follows that there exists $\kappa>0$ such that $\Sigma_{\rho_{k}}^{(i)}
    \preceq\kappa I$ for any $i$ and $k$.
    In addition, Lemma\ref{lem:inequalities for condition where optimal covariance matrices are full-rank} gives $\Gamma_k\preceq\hat\Gamma_k$.
    Thus, $\Sigma_{\rho_{k}}+\Gamma_{k} \preceq \kappa I+\hat{\Gamma}_{k} \preceq c_{k} I$, where $c_{k}:=
    \kappa+\lambda_{\max}(\hat{\Gamma}_{k})>0$.
    Consequently, $L_k=(\Sigma_{\rho_k}+\Gamma_k)^{-1}\succeq c_k^{-1}I$.
    It hence follows that $ L_{k}(\check{M}_{k}+\Gamma_{k})L_{k} \succeq \gamma_{k} L_{k}^{2} \succeq \frac{\gamma_{k}}{c_{k}^{2}}I$,
    and thus
    \begin{align*}
        \Sigma_{\rho_{k}}^{+} \succeq \frac{\gamma_{k}}{c_{k}^2}\Sigma_{\rho_{k}}^2.
    \end{align*}
    Together with \eqref{eq:gamma theorem 3}, this yields
    \begin{align}
        \lambda_{\max}(\Sigma_{\rho_{k}}^{+}) &\geq \frac{\gamma_{k}}{c_{k}^{2}} \lambda_{\max}(\Sigma_{\rho_{k}})^2 \geq \frac{\gamma_{k}^{3}}{c_{k}^{2}}>0. \label{eq:second case theorem 3}
    \end{align}

    From \eqref{eq:first case theorem 3} and \eqref{eq:second case theorem 3}, we obtain
    \begin{align*}
        \lambda_{\max}(\Sigma_{\rho_{k}}^{(i)}) \geq \min\left\{ \lambda_{\max}(\Sigma_{\rho_{k}}^{(0)}), \frac{\gamma_{k}^{3}}{c_{k}^{2}}\right\}>0
    \end{align*}
    for any $i\in\mathbb Z_{\geq 0}$.
    Combining this with Proposition \ref{prop:algorithm converges to equilibrium point} completes the proof.

\section{Proof of Theorem \ref{thm:condition where algorithm converges to zero covariance matrix}} \label{app:Proof of Theorem of condition where algorithm converges to zero covariance matrix}

    Let $\rho^{(\infty)}\in\mathcal E$ be an arbitrary cluster point of the sequence generated by Algorithm \ref{alg:alternating optimization algorithm for MIOCPs}.
    By Proposition \ref{prop:algorithm converges to equilibrium point}, $\rho^{(\infty)}$ is a fixed point of Algorithm \ref{alg:alternating optimization algorithm for MIOCPs}, and thus its covariance matrices satisfy
    \begin{align}
        \Sigma_{\rho_{k}}L_{k}\left(E_{k}\Sigma_{x_{k}}E_{k}^{\top}-\Sigma_{\rho_{k}}-\Gamma_{k}\right)L_{k}\Sigma_{\rho_{k}}=0 \label{eq:fixed point theorem 4}
    \end{align}
    for any $k\in\llbracket0,T-1\rrbracket$, where the quantities in \eqref{eq:fixed point theorem 4} are evaluated at $\rho^{(\infty)}$.

    We first consider $k=0$.
    Since $\Sigma_{x_{0}}=\Sigma_{x_{0}}^{\mathrm{zero}}=\Sigma_{x_{\mathrm{ini}}}$, the bound used in the proof of Theorem \ref{thm:condition where optimal covariance matrices are 0} gives $E_{0}\Sigma_{x_{0}}E_{0}^{\top}-\Gamma_{0} \preceq \hat{M}_{0} \prec0$.
    It hence follows that $E_{0}\Sigma_{x_{0}}E_{0}^{\top}-\Sigma_{\rho_{0}} -\Gamma_{0} \prec 0$.
    Suppose that $\Sigma_{\rho_{0}}\neq 0$.
    Since $L_{0}\succ0$, we have $L_{0}\Sigma_{\rho_{0}}\neq0$, and hence $\Sigma_{\rho_{0}}L_{0}\left(E_{0}\Sigma_{x_{0}}E_{0}^{\top}-\Sigma_{\rho_{0}}-\Gamma_{0}\right)L_{0}\Sigma_{\rho_{0}}\neq0$, which contradicts \eqref{eq:fixed point theorem 4}.
    Therefore, $\Sigma_{\rho_0}=0$.
    By Proposition \ref{prop:optimal policy of MIOCP for fixed prior}, $\Sigma_{\rho_{0}}=0$ implies that the limiting policy at $k=0$ is $\mathcal N(0,0)$.
    Under this policy, $\Sigma_{x_{1}}=\Sigma_{x_{1}}^{\mathrm{zero}}$.
    Applying the same argument to $k=1$ gives $\Sigma_{\rho_{1}}=0$.
    Proceeding recursively, we have $\Sigma_{\rho_k}=0$ for any $ k\in\llbracket0,T-1\rrbracket$, which completes the proof.

\bibliographystyle{plain}        
\bibliography{Automatica_MIOCP}  

\end{document}